\documentclass[12pt, reqno]{amsart}
\usepackage{graphicx, amssymb, amsthm, amsfonts, latexsym, epsfig, amsmath, amscd, amsbsy}

\allowdisplaybreaks

\usepackage[usenames]{color}
\usepackage{tikz, caption, subcaption}
\usepackage{enumerate}
\usepackage{epstopdf}

\usetikzlibrary{decorations.pathmorphing}
\usetikzlibrary{calc}

\usetikzlibrary{arrows, patterns, calc}
\pgfdeclarelayer{bg}
\pgfsetlayers{bg,main}

\pgfdeclarelayer{bg}
\pgfsetlayers{bg,main}

\numberwithin{equation}{section}
\theoremstyle{plain}
\newtheorem{theorem}{Theorem}[section]
\newtheorem{lem}[theorem]{Lemma}

\theoremstyle{definition}
\newtheorem{df}[theorem]{Definition}
\newtheorem{rem}[theorem]{Remark}

\def\Int{\mathop\mathrm{Int}}

\def\dist{\mathop\mathrm{dist}}

\usepackage{hyperref}
\hypersetup{
    colorlinks=true,
    citecolor=violet,
    pdftitle={Tangential homoclinic points for Lozi maps},
    pdfauthor={Kristijan Kilassa Kvaternik}
    }

\begin{document}
	
	\title{Tangential homoclinic points for Lozi maps}
	\author{Kristijan Kilassa Kvaternik}
	\address[K. Kilassa Kvaternik]{
		University of Zagreb,
		Faculty of Electrical Engineering and Computing,
		Department of Applied Mathematics,
		Unska 3,
		10\,000 Zagreb,
		Croatia -- and -- Jagiellonian University,
		Institute of Mathematics,
		ul.\ prof.\ Stanis{\l}awa {\L}ojasiewicza 6,
		30-348 Kraków,
		Poland}
	\urladdr{\url{https://www.fer.unizg.hr/en/kristijan.kilassa_kvaternik}}
	\email{kristijan.kilassakvaternik@fer.unizg.hr}
	\thanks{This work was supported in part by the Croatian Science Foundation grant HRZZ-MOBODL-2023-08-4960, and in part by the Croatian Science Foundation grant HRZZ-IP-2022-10-9820 GLODS}
		
	\date{December 17, 2024}
	
	
	\begin{abstract}
		
	For the family of Lozi maps, we study homoclinic points for the saddle fixed point $X$ in the first quadrant. Specifically, in the parameter space, we examine the boundary of the region in which homoclinic points for $X$ exist. For all parameters on that boundary, all intersections of the stable and unstable manifold of $X$, apart from $X$, are tangential, or these manifolds intersect along a segment. We ultimately prove that for such parameters, all possible homoclinic points for $X$ are iterates of two special points $Z$ and $V$, or iterates of points on a segment joining $V$ with an iterate of $Z$. Additionally, we describe the parameter curves that form the boundary and provide explicit equations for several of them.   
		
	\end{abstract}

	\maketitle	
	
	\baselineskip=18pt
	
	\section{Introduction}\label{sec:intro}

	With the intention of providing a map that would be an analogue of the H\'{e}non map, but at the same time possess simpler properties, R.\ Lozi introduced in his 1978 paper \cite{lozi1978attracteur} a two-parameter family of homeomorphisms
	\begin{equation*}
	L_{a,b}\colon\mathbb{R}^2\rightarrow\mathbb{R}^2,\quad L_{a,b}(x,\:y)=(1+y-a|x|,\: bx),
	\end{equation*}
where $a,b\in\mathbb{R}$ are the parameters. This family has since been known as the \emph{Lozi map family}.

	Despite the piecewise af\mbox{}fine nature of the Lozi map, it turned out that it induces complicated dynamics. Throughout the studies of that dynamics, in the case when $b>0$, much attention has been paid to the stable and unstable manifold of the fixed point $X$ in the first quadrant, $W_X^s$ and $W_X^u$. These notions are formally introduced in the next section. In this paper, we are interested in points at which $W_X^s$ and $W_X^u$ intersect. Such points are called \emph{homoclinic points} for the fixed point $X$. In general, the presence of homoclinic points indicates complex dynamical behavior which is often linked to chaos. 
	
	Specifically, we focus on the border case of the existence of homoclinic points for $X$. More formally, in the parameter space, let $\mathcal{H}$ denote the set of all parameters $(a,b)$, $a>0$, $0<b<1$, such that the Lozi maps $L_{a,b}$ exhibit homoclinic points for $X$. We study $\partial\mathcal{H}$, the boundary of $\mathcal{H}$, which we call the \emph{boundary of existence of homoclinic points} for $X$. A part of this boundary is represented in Figure \ref{figure:homoclinic_border_curves}. As it will follow from Lemma \ref{lemma:hom_tang}, for parameter values on the boundary $\partial\mathcal{H}$, all homoclinic intersections of $W_X^s$ and $W_X^u$ are either \emph{tangential}, which is illustrated in Figure \ref{figure:homoclinic_border_WsWu_config}, or $W_X^s$ and $W_X^u$ intersect along a segment as in Figure \ref{fig:homoclinic_segments_example}. Here, the notions of tangential and transverse intersections will be readjusted since $W_X^s$ and $W_X^u$ are polygonal lines and therefore not dif\mbox{}ferentiable curves. In the main result of the paper, we determine the set of homoclinic points for $X$ in the border case. In this context, the two distinct points $Z$ and $V$, at which $W_X^u$ and $W_X^s$, respectively, intersect the $x$- and $y$-axis for the first time, starting from $X$, proved to be of particular importance. For an exact definition of the points $Z$ and $V$, we refer the reader to the next section; specially, Figure \ref{fig:Lozi_stable_unstable_Z_V} illustrates several iterates of these points as an aid for visual intuition.
	
	\begin{theorem} \label{thm:hom_T0V0}
  		For parameters on the boundary $\partial\mathcal{H}$, the set of all homoclinic points for $X$ is one of the following:
  		\begin{enumerate}
  		\item the orbit of $Z$,
  		\item the orbit of $V$,
  		\item the union of orbits of $Z$ and $V$,
  		\item the union of orbits of points on a segment joining $V$ with an iterate of $Z$.
  		\end{enumerate}
	\end{theorem}
	
	A rather surprising fact here is that $W_X^s$ and $W_X^u$ can also intersect along a segment of points, even more so since this cannot happen for parameter pairs $(a,b)$ such that $0<b<1$, $a>0$ and $a>b+1$; see Remark \ref{rem:misiurewicz_no_segments}.
	
	It is worth mentioning that our main theorem, in a way, represents a counterpart to the result of Ishii in \cite{ishii1997towards1}, where the first tangency problem was solved. There, the author showed that for all parameters on the boundary of the maximal entropy region (the region where $L_{a,b}$ is conjugate to the full shift on two symbols), $L_{a,b}$ has a heteroclinic tangency for $b>0$, i.e., a homoclinic tangency for $b<0$, and these tangencies are iterates of a special tangency on the $x$-axis in the phase space. In the parameter space, as we leave the maximal entropy region and move along the positive $a$-axis, from greater to smaller values, at one moment we pass the boundary of existence of homoclinic points for $X$, $\partial\mathcal{H}$, and enter the region where these points cease to exist. Therefore, the points stated in Theorem \ref{thm:hom_T0V0} can be considered as the \emph{points of last tangency} for $L_{a,b}$.
	
	Moreover, the boundary $\partial\mathcal{H}$ also appears in the studies of the \emph{zero entropy locus} of $L_{a,b}$. Firstly, Yildiz provided in \cite{yildiz2011monotonicity} partial results about zero topological entropy $h_{top}$ for $L_{a,b}$. In addition, the author numerically observed a region in the parameter space in which $h_{top}(L_{a,b})=0$. That region is in part bounded by $\partial\mathcal{H}$, and the author remarks that this boundary is expected to be piecewise algebraic. This last claim is justified by the fact that the curves forming $\partial\mathcal{H}$ represent intersections of $W_X^s$ and $W_X^u$, which are piecewise linear. This is explained in greater detail in Section \ref{sec:examples_bd_curves}.
	
	Secondly, very recent results of Misiurewicz and \v{S}timac in \cite{misiurewicz2024zero} expand those of Yildiz and show that $h_{top}(L_{a,b})=0$ in a large region of parameters that is contained in the complement of $\mathcal{H}$. Thirdly, Burns and Weiss proved in \cite{burns1995geometric} that a dif\mbox{}feomorphism possessing a transverse homoclinic point has positive topological entropy. 
	
	Here, it is important to mention that in a private correspondence, Y.\ Ishii and D.\ Sands brought to the author's attention the existence of parameter pairs for which the Lozi map exhibits transverse homoclinic points for period 6 saddle points. Even though these findings indicate that the zero entropy locus of $L_{a,b}$ is a proper subset of the complement of $\mathcal{H}$, the study of the dynamical behavior near $\partial\mathcal{H}$ will be of relevance in determining the exact zero entropy locus of the Lozi family.
	
	Finally, the aforementioned result of Burns and Weiss also demonstrates another incentive to study $\partial\mathcal{H}$: namely, homoclinic points have so far been an object of research mainly in the context of di\-f\mbox{}fe\-o\-mor\-phisms, that is, one required some degree of dif\mbox{}ferentiability. This work thus presents a step towards a comparable theory for homeomorphisms of the plane.
	
	This paper is organized as follows: in Section \ref{sec:prelim}, we review the preliminary notions on Lozi maps and give an overview of the notation. The following Section \ref{sec:tan_hom_pts} is divided into three parts. In Subsection \ref{subsec:stable_struct}, we describe the zigzag structure that the stable manifold $W_X^s$ forms in the third quadrant of the plane. Subsection \ref{subsec:border_hom_pts} is dedicated to tangential homoclinic points. The proof of the main theorem is concluded in Subsection \ref{subsec:hom_intersec_segment}, where we consider homoclinic intersections along segments. Lastly, in Section \ref{sec:examples_bd_curves}, we describe the curves that form the boundary of existence of homoclinic points for $X$ and explicitly state the equations for some of them in Appendix \ref{appendix:curves_Cn_equations}.
	
	This paper is based on Chapter 2 of the author's PhD dissertation \cite{kilassa2022tangential}. The author of the paper would like to thank S.\ \v{S}timac for the support, reading the manuscript, and giving valuable suggestions for its improvement, as well as J.\ Boro\'nski for the constructive discussions related to specific proofs in this work. The author also wishes to thank Y.\ Ishii and D.\ Sands for pointing out new findings that refine the current understanding of the zero entropy locus of the Lozi family, and, by extension, the role of the boundary $\partial\mathcal{H}$ in this context. Finally, the author wishes to express his sincere gratitude to the anonymous reviewer whose constructive comments strengthened the argumentation of this paper and led to careful considerations of homoclinic intersections along a segment.

	\section{Preliminaries}\label{sec:prelim}

	\subsection{Lozi maps}\label{subsec:lozi_maps}
	
	To study homoclinic points in the case when the Lozi map $L_{a,b}$ is orientation-reversing (i.e., $b>0$), one can observe parameter pairs $(a,b)$ for which $0 < b < 1$ and $a+b > 1$. Namely, if $L_{a,b}$ is a Lozi map with $b>1$, then it is conjugate to $L^{-1}_{a/b,\,1/b}$. Moreover, if $b-1 \geqslant a$, then $L_{a,b}$ does not have any fixed points. If $b-1 < a \leqslant 1-b$, then $L_{a,b}$ has one fixed point $X$ in the first quadrant. This point is attracting for $a < 1-b$, and it is not hyperbolic when $a = 1-b$. 

	For parameter pairs $(a,b)$ such that $0<b<1$ and $a+b>1$, the Lozi map $L_{a,b}$ has two hyperbolic saddle fixed points, $X=\bigl(\frac{1}{1+a-b},\:\frac{b}{1+a-b}\bigr)$ in the first and $Y=\bigl(\frac{1}{1-a-b},\:\frac{b}{1-a-b}\bigr)$ in the third quadrant. The eigenvalues of the dif\mbox{}ferential of $L_{a,b}$ at $X$ are
	\begin{equation*}
	\lambda_X^u = \tfrac{1}{2}\left(-a-\sqrt{a^2+4b}\right),\quad \lambda_X^s = \tfrac{1}{2}\left(-a+\sqrt{a^2+4b}\right),
	\end{equation*} 
while those at $Y$ are given by
	\begin{equation*}
	\lambda_Y^u = \tfrac{1}{2}\left(a+\sqrt{a^2+4b}\right),\quad \lambda_Y^s = \tfrac{1}{2}\left(a-\sqrt{a^2+4b}\right).
	\end{equation*}
Observe that $\lambda_X^u<-1$, $0<\lambda_X^s<1$ and $\lambda_Y^u>1$, $-1<\lambda_Y^s<0$. Moreover, for every eigenvalue $\lambda$, the corresponding eigenvector is given by $\binom{\lambda}{b}$.

	For every point $A \in \mathbb{R}^2$ and every $k \in \mathbb{Z} \setminus \{0\}$, we put $A^k=L_{a,b}^k(A)$; specially, $A^0=A$. Recall that the set $\{A^k \colon k \in \mathbb{Z}\}$ is called the \emph{orbit} of $A$. Moreover, a set $\mathcal{B} \subseteq \mathbb{R}^2$ is $L_{a,b}$-\emph{invariant} if $L_{a,b}(\mathcal{B}) \subseteq \mathcal{B}$. An $L_{a,b}^{-1}$-\emph{invariant} set is defined analogously.

	Recall that the \emph{unstable manifold} of $X$ is the set $W_X^u = \{T \in \mathbb{R}^2 \colon T^{-n} \overset{n \rightarrow \infty}{\longrightarrow} X\}$. Similarly, the \emph{stable manifold} of $X$ is $W_X^s = \{T \in \mathbb{R}^2 \colon T^n \overset{n \rightarrow \infty}{\longrightarrow} X\}$. We know that $W_X^u$ and $W_X^s$ are $L_{a,b}$- and $L_{a,b}^{-1}$-invariant sets which contain $X$. As already explained in \cite{misiurewicz1980strange}, $W_X^u$ and $W_X^s$ are not manifolds in the true meaning of that notion since they are polygonal lines in the plane; see Figure \ref{fig:Lozi_stable_unstable_Z_V}. We nevertheless follow the established terminology and call them manifolds.

	\begin{figure}
	\begin{center}
		\includegraphics[width=\linewidth]{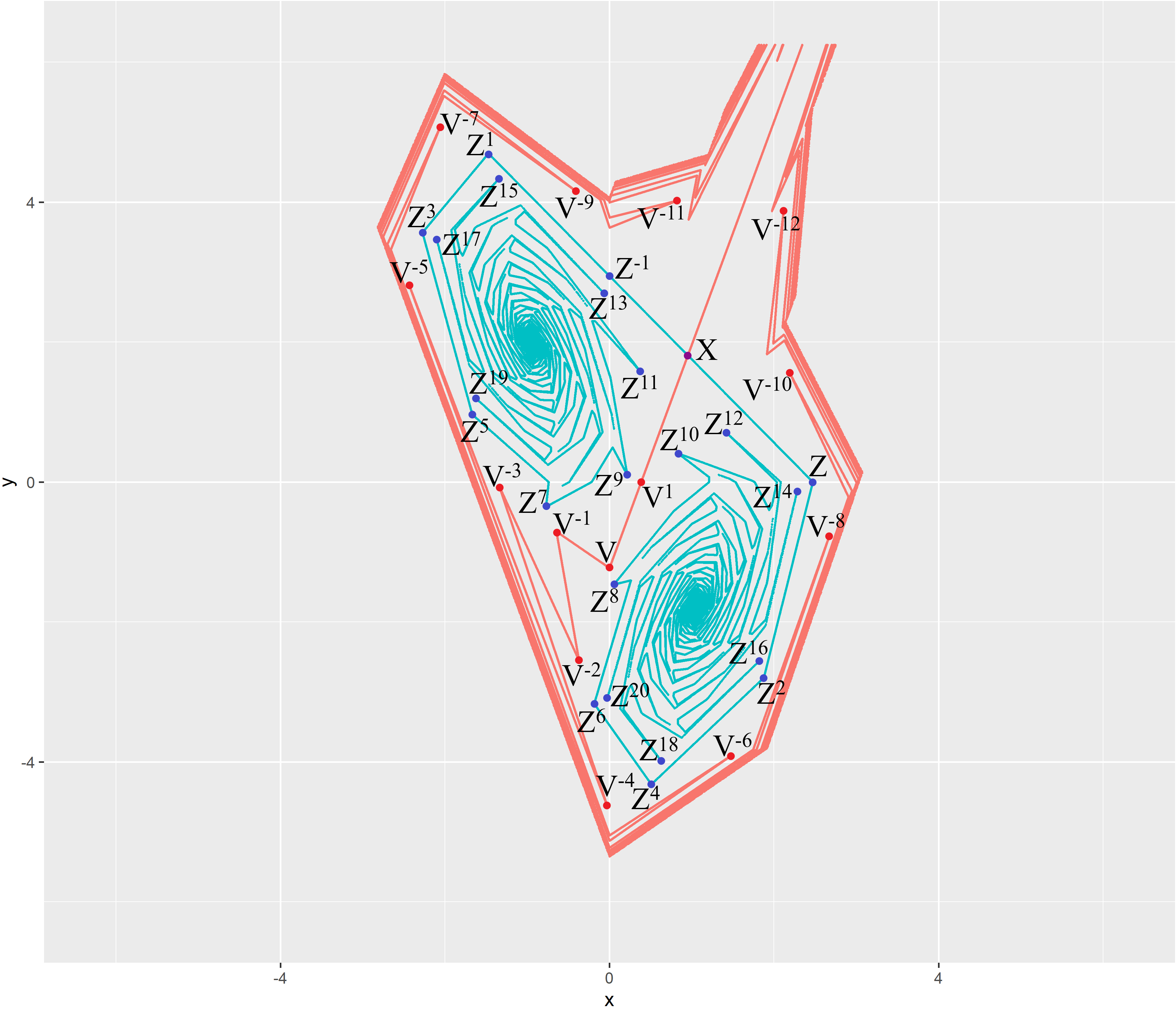}
	\end{center}
	\caption{The stable (red) and unstable (blue) manifold of $X$ for parameter values $a=1$, $b=0.95$, together with some iterates of $Z$ and $V$.}
	\label{fig:Lozi_stable_unstable_Z_V}
	\end{figure}
	
	Observe the unstable manifold $W_X^u$. We denote the half of that manifold that starts at $X$ and goes to the right by $W_X^{u+}$. That half intersects the positive $x$-axis for the first time at the point
	\begin{equation}
	Z=\left(\frac{2+a+\sqrt{a^2+4b}}{2(1+a-b)},\: 0\right)=\left(\frac{2}{2+a-\sqrt{a^2+4b}},\: 0\right).
	\label{eq:Z}
	\end{equation}
The other half, starting at $X$ and going to the left, we denote by $W_X^{u-}$. Note that
	\begin{equation*}
	W_X^{u+} = \{X\}\cup\bigcup_{n=-\infty}^{\infty}L_{a,b}^{2n}\bigl(\overline{ZZ^2}\bigr) ,\quad W_X^{u-} = \{X\}\cup\bigcup_{n=-\infty}^{\infty}L_{a,b}^{2n}\bigl(\overline{Z^{-1}Z^1}\bigr).
	\end{equation*}
Similarly, we denote by $W_X^{s-}$ the half of the stable manifold $W_X^s$ that starts at $X$ and goes down. That part intersects the negative $y$-axis for the first time at the point
	\begin{equation}
	V=\left(0,\: \frac{2b - a - \sqrt{a^2 + 4b}}{2 ( 1 + a - b)}\right)=\left(0,\: -\frac{2b}{-a+2b+\sqrt{a^2+4b}}\right).
	\label{eq:V}
	\end{equation}
We have
	\begin{equation*}
	W_X^{s-} = \{X\}\cup\bigcup_{n=-\infty}^{\infty}L_{a,b}^{-n}\bigl(\overline{VV^1}\bigr).
	\end{equation*}
The other half of $W_X^s$ is a half-line emanating from $X$ and going up in the first quadrant. That part we denote by $W_X^{s+}$.

	As already stated, $W_X^{u}$ and $W_X^{s}$ are polygonal lines in the plane. We are interested in the points at which these lines break, i.e., the endpoints of maximal straight line segments contained in $W_X^{u}$ and $W_X^{s}$. Adopting the terminology from \cite{boronski2023densely}, we call such points \emph{post-critical points} if they lie on $W_X^u$ and \emph{V-points} if they are on $W_X^s$.

	\subsection{Notation}\label{subsec:notation}
	
	Since we are working in the Euclidean plane, we will use the following notation to denote geometrical objects and their topological characteristics.
	\begin{itemize}
	\item The first, second, third and fourth quadrant of the Cartesian coordinate system will be denoted by $\mathcal{Q}_1$, $\mathcal{Q}_2$, $\mathcal{Q}_3$ and $\mathcal{Q}_4$, respectively.
	\item Points in the plane will be denoted by capital Latin letters, possibly with indices: $A,B,C,\ldots,A_1,B_1,C_1,\ldots$ One exception is $L_{a,b}$ which will denote the Lozi map. 
	\item Specially, $X$ and $Y$ denote the fixed points of $L_{a,b}$. Moreover, $Z$ and $V$ will denote points on $W_X^u$ and $W_X^s$, respectively, as defined in Subsection \ref{subsec:lozi_maps}; see Figure \ref{fig:Lozi_stable_unstable_Z_V}.
	\item For a point $A\in\mathbb{R}^2$, $A_x$ and $A_y$ are the $x$- and $y$-coordinate of $A$, respectively.
	\item For $A,B\in\mathbb{R}^2$, the straight line segment with endpoints $A$ and $B$ will be denoted by $\overline{AB}$.
	\item For $A\in\mathbb{R}^2$ and $\varepsilon>0$, $B_{\varepsilon}(A)$ represents the open ball in the plane centered at $A$ of radius $\varepsilon$.
	\item Lowercase Greek letters $\alpha,\beta,\gamma,$ etc.\ will stand for line segments, straight or polygonal ones.
	\item Capital script letters $\mathcal{A},\mathcal{B},$ etc.\ will denote two-dimensional subsets of the plane, typically polygons.
	\item For $\mathcal{A}\subset\mathbb{R}^2$, we will use the following notation:
		\begin{itemize}
			\item $\Int\mathcal{A}\ \ldots$ interior of $\mathcal{A}$,
			\item $\partial\mathcal{A}\ \ldots$ boundary of $\mathcal{A}$.
		\end{itemize}
	\end{itemize}			
	
	Finally, we will use specific notation concerning the Lozi map and the stable and unstable manifolds $W_X^s$ and $W_X^u$.
	\begin{itemize}
		\item For every $k\in\mathbb{Z}$ and every point $A \in \mathbb{R}^2$, we put $A^k=L_{a,b}^{k}(A)$. Specially, $A^{0} = A$.
		\item For points $A,B\in W_X^u$, we put:
			\begin{itemize}
				\item $[A,B]^{u}\subset W_X^u\ \ldots$ polygonal line lying on $W_X^{u}$ with $A$ and $B$ as endpoints,
				\item specially, if $[A,B]^{u}$ is a straight line segment, we denote it by $\overline{AB}^{u}$,
				\item $[A,B)^{u}:=[A,B]^{u}\setminus\{B\}$,
				\item $(A,B]^{u}:=[A,B]^{u}\setminus\{A\}$,
				\item $(A,B)^{u}:=[A,B]^{u}\setminus\{A,B\}$.
			\end{itemize}
		\item For $A,B\in W_X^s$, we define the sets $[A,B]^{s}$, $\overline{AB}^{s}$, $[A,B)^{s}$, $(A,B]^{s}$ and $(A,B)^{s}$ analogously.			 
	\end{itemize}

	\section{Boundary homoclinic points}\label{sec:tan_hom_pts}
	
	We will investigate homoclinic points for the fixed point $X$ of the Lozi map; more precisely, we will describe the boundary of the area of their existence in the parameter space. Recall that a point $T$, $T\neq X$, is said to be a \emph{homoclinic point} for $X$ if $T$ is contained in the intersection of its stable and unstable manifold, $T\in W^{s}_{X}\cap W^{u}_{X}$. The main result is that in the border case, all homoclinic points for $X$ are iterates of the points $Z$ and/or $V$; see Theorem \ref{thm:hom_T0V0}. In this section we will observe parameter pairs $(a,b)$ such that $0<b<1$ and $a+b>1$.

	\subsection{Structure of the stable manifold}\label{subsec:stable_struct}
	
	In this subsection, we will describe in greater detail the shape that the stable manifold $W_X^{s}$ forms in the third quadrant which will give us some insight into the order of appearance of V-points on it. The following technical lemma will be useful during that analysis.
	
	\begin{lem}[Some geometric properties of $L_{a,b}$ and $L_{a,b}^{-1}$]\label{lemma:lozi_geom}\mbox{}
	\begin{enumerate}
		\item The image of the $y$-axis under $L_{a,b}$ is the $x$-axis and the image of the $x$-axis under $L_{a,b}$ is the curve $x=1-\frac{a}{b}|y|$.
		\item The image of the $x$-axis under $L_{a,b}^{-1}$ is the $y$-axis and the image of the $y$-axis under $L_{a,b}^{-1}$ is the curve $y=a|x|-1$.
		\item Let $\alpha$ be a straight line segment in the lower half-plane which lies on a straight line whose slope equals $s_1$. Then $L_{a,b}^{-1}(\alpha)$ lies on a straight line whose slope equals
		\begin{equation*}
		s_2=b\cdot\frac{1}{s_1}-a.
		\end{equation*}
	\end{enumerate}
	\end{lem}
	
	We know that $V$ is the intersection of the $y$-axis and a straight line through $X$ whose slope equals $s_0=\frac{b}{\lambda_X^s}=\frac{1}{2}\left(a+\sqrt{a^2+4b}\right)$. Furthermore,
	\begin{equation*}
	W_X^{s-}=\overline{XV^1}^{s}\cup\bigcup_{n=0}^{\infty}L_{a,b}^{-n}\bigl(\overline{VV^1}^{s}\bigr),
	\end{equation*}
	so we see that the part of $W_X^{s-}$ which does not contain $\overline{XV}^{s}$ consists of preimages of a line segment lying in the lower half-plane. Notice that Lemma \ref{lemma:lozi_geom}(3) gives a recurrence for the slopes of certain parts of those preimages.
	
	\begin{lem}\label{lem:lozi_recurr_coef}
	The sequence $(s_n)_{n\in\mathbb{N}_0}$ given by the recurrence
	\begin{equation}\label{eg:recurr}
	s_{n+1}=b\cdot\frac{1}{s_n}-a,\ n\geqslant0;\quad s_0=\tfrac{1}{2}\left(a+\sqrt{a^2+4b}\right),
	\end{equation} 
	has the following properties:
	\begin{enumerate}
	\item if $s_{n_1}>0$ for some $n_1\in\mathbb{N}_0$, then $s_n>0$ for all $n\leqslant n_1$,
	\item if $s_{n_2}<0$ for some $n_2\in\mathbb{N}_0$, then $s_n<0$ for all $n\geqslant n_2$,
	\item for all $n\in\mathbb{N}_0$, if $s_{2n+1}>0$, then $s_{2n+2}>0$,
	\item $(s_n)$ converges and $\displaystyle\lim_{n\rightarrow\infty}s_n=\tfrac{1}{2}\left(-a-\sqrt{a^2+4b}\right)$.  
	\end{enumerate}
	\end{lem}
	
	\begin{proof}
	The first two claims immediately follow from recurrence (\ref{eg:recurr}) and $a,b>0$. In order to prove the third claim, let
	\begin{equation*}
	M_1:=\tfrac{1}{2}\left(-a-\sqrt{a^2+4b}\right),\quad M_2:=\tfrac{1}{2}\left(-a+\sqrt{a^2+4b}\right)
	\end{equation*}
	be the roots of the equation $M^2+aM-b=0$, i.e., the fixed points of the function
	\begin{equation*}
	f\colon\mathbb{R}\setminus\{0\}\rightarrow\mathbb{R},\quad f(x)=b\cdot\frac{1}{x}-a,
	\end{equation*}	 
	which defines recurrence (\ref{eg:recurr}). Because $a,b>0$, we see that $M_1<0$, $M_2>0$ and $|M_2|<|M_1|$. By setting
	\begin{equation*}
	j_n:=\frac{s_n-M_1}{s_n-M_2},\quad n\in\mathbb{N}_0,
	\end{equation*}
	we see that
	\begin{equation*}
	j_{n+1}\stackrel{(\ref{eg:recurr})}{=}\frac{b\cdot\frac{1}{s_n}-a-M_1}{b\cdot\frac{1}{s_n}-a-M_2}=\frac{b\cdot\frac{1}{s_n}-b\cdot\frac{1}{M_1}}{b\cdot\frac{1}{s_n}-b\cdot\frac{1}{M_2}}=\frac{M_2}{M_1}\cdot j_n
	\end{equation*}	 
	for every $n\in\mathbb{N}_0$. If $\mu:=\frac{M_2}{M_1}$, we have $-1<\mu<0$ and $(j_n)_{n\in\mathbb{N}_0}$ satisfies the recurrence
	\begin{equation}
	\label{e:j_n_recurrence}
	j_{n+1}=\mu j_n,\ n\geqslant0;\quad j_0=\tfrac{1}{a}\left(a+\sqrt{a^2+4b}\right)>0.
	\end{equation}
	Therefore, $j_{2n}>0$ and $j_{2n+1}<0$ for all $n\in\mathbb{N}_0$. Now assume that $s_{2n+1}>0$ for some $n\in\mathbb{N}_0$. Due to the fact that
	\begin{equation}
	\label{e:k_nj_n}
	s_n=\frac{M_1-M_2j_n}{1-j_n},\quad n\in\mathbb{N}_0,
	\end{equation}
	and since $j_{2n+1}<0$, i.e., $1-j_{2n+1}>0$, it follows that $M_1-M_2j_{2n+1}>0$. A direct calculation now gives
	\begin{equation*}
	s_{2n+2}\stackrel{(\ref{e:k_nj_n})}{=}\frac{M_1-M_2j_{2n+2}}{1-j_{2n+2}}\stackrel{(\ref{e:j_n_recurrence})}{=}\frac{M_1-M_2\mu j_{2n+1}}{1-\mu j_{2n+1}}=\frac{M_1^2-M_2^2j_{2n+1}}{M_1-M_2j_{2n+1}},
	\end{equation*}
	which combined together with $M_1^2-M_2^2j_{2n+1}>0$ yields $s_{2n+2}>0$. 
	
	Finally, since $\vert\mu\vert<1$, we see that $\displaystyle\lim_{n\rightarrow\infty} j_n=\lim_{n\rightarrow\infty}\mu^{n}j_0=0$. Consequently,
	\begin{equation*}
	\lim_{n\rightarrow\infty}s_n\stackrel{(\ref{e:k_nj_n})}{=}\lim_{n\rightarrow\infty}\frac{M_1-M_2j_n}{1-j_n}=M_1,
	\end{equation*}
	which proves the fourth claim.
	\end{proof}
	
	\begin{figure}[!ht]
	\begin{center}
	\begin{tikzpicture}[auto]
			\tikzstyle{nodec}=[draw,circle,fill=black,minimum size=2pt,
			inner sep=0pt, label distance=2mm]
			\tikzstyle{nodeh}=[draw,circle,fill=white,minimum size=4pt,
			inner sep=0pt]
			\tikzstyle{dot}=[circle,draw=none,fill=none,minimum size=0pt,inner sep=2pt, outer sep=-1pt]

			\draw[->] (-8.5,0)--(4,0) node [below]{$x$};
			\draw[->] (0,-8)--(0,1) node [left]{$y$};
			\node[label={[xshift=-0.2cm, yshift=-0.7cm]$0$}] at (0,0) {};
			
			\coordinate (e0) at (0,0);
			\coordinate (ex) at (1,0);
			\coordinate (ey) at (0,1);
			
			\node[nodec, color=red, label={[right, color=red]$V$}] (v0) at (0,-1) {};
			\node[nodec, color=red, label={[color=red]$V^{-1}$}] (v1) at (-1,-1.25) {};
			\node[nodec, color=red, label={[above, color=red]$V^{-2}$}] (v2) at (-2,-2) {};
			
			\node[nodec, color=red, label={[below right, color=red]$V^{-i_0+2}$}] (vi0m2) at (-3,-3) {};
			\node[nodec, color=red, label={[below, color=red]$V^{-i_0+1}$}] (vi0m1) at (-4,-3.5) {};
			\node[nodec, color=red, label={[color=red]$V^{-i_0}$}] (vi0) at (-6,-2) {};
			\node[nodec, color=red, label={[right, color=red]$V^{-i_0-1}$}] (vi0p1) at (-4.5,-5.5) {};
			\node[nodec, color=red, label={[below, color=red]$V^{-n_0+1}$}] (vn0m1) at (-4.25,-7) {};
			\node[nodec, color=red, label={[color=red]$V^{-n_0}$}] (vn0) at (-8,0.5) {};
			
			\draw[red, thick] (v0) to node[dot]{$\alpha_1$} (v1);
			\draw[red, thick] (v1) to node[dot]{$\alpha_2$} (v2);
			\draw[red, thick, dashed] (v2)--(vi0m2);
			\draw[red, thick] (vi0m2) to node[dot, swap]{$\alpha_{i_0-1}$} (vi0m1);
			\draw[red, thick] (vi0m1) to node[dot, swap]{$\alpha_{i_0}$} (vi0);
			\draw[red, thick] (vi0) to node[dot]{$\alpha_{i_0+1}$} (vi0p1);
			\draw[red, thick, dashed] (vi0p1)--(vn0m1);
			\draw[red, thick] (vn0m1) to node[dot]{$\alpha_{n_0}$} (vn0);
			
			\coordinate (q1) at (intersection of vi0m1--vi0m2 and e0--ex);
			\coordinate (lq1) at (intersection of vi0--vi0m1 and e0--ey);
			
			\node[nodec, label={[above]$T_1$}] at (q1) {};
			\node[nodec, label={[right]$T_1^{-1}$}] at (lq1) {};
			
			\draw[dotted] (vi0m2)--(q1); 
			\draw[dotted] (vi0m1)--(lq1);

			\end{tikzpicture}
	\end{center}
	\caption{The figure illustrates the proof of Lemma \ref{lem:zigzag}. The zigzag part of the stable manifold $W_X^s$ is represented in red.}
	\label{figure:stable_zigzag}
	\end{figure}

	Recall that the $i$th quadrant of the Cartesian coordinate system in the plane is denoted by $\mathcal{Q}_i$, for $i=1,2,3,4$.	
	
	\begin{lem}[Zigzag structure of $W_X^{s}$ in the third quadrant]\label{lem:zigzag}
	There exists a positive integer $n$ such that $V^{-n}$ lies in $\mathcal{Q}_2$. The smallest such positive integer $n_0$ is odd, $[V,V^{-n_0+1}]^{s}$ is contained in $\mathcal{Q}_3$ and all V-points of $W_X^{s}$ on it are negative iterates of $V$.
	\end{lem}
	
	\begin{proof}
	Observe the backward orbit $(V^{-n})_{n\in\mathbb{N}_0}$ of $V$. We know that $V$ lies on the negative $y$-axis. In general, if a point lies in $\mathcal{Q}_3$, then its image under $L_{a,b}^{-1}$ is contained in $\mathcal{Q}_2$ or $\mathcal{Q}_3$ below the curve $y=a\vert x\vert-1$. For every $n\in\mathbb{N}$, let $\alpha_n :=[V^{-n+1},V^{-n}]^{s}$ and notice that $\alpha_{n+1}=L_{a,b}^{-1}(\alpha_n)$. Since $L_{a,b}^{-1}$ acts as an af\mbox{}fine map in the lower half-plane, we see that $\alpha_1=L_{a,b}^{-1}\bigl(\overline{V^1V}^{s}\bigr)$ is a straight line segment, and inductively, if $\alpha_n$ is a straight line segment contained in $\mathcal{Q}_3$, then $\alpha_{n+1}$ is again a straight line segment which may intersect the $x$-axis. If it intersects the $x$-axis, it follows that $\alpha_{n+2}$ (and consequently all $\alpha_{i}$ for $i\geqslant n+2$) intersects both coordinate axes in $\mathcal{Q}_3$; more precisely, its intersection with $\mathcal{Q}_3$ is a straight line segment whose one endpoint lies on the $x$-axis and the other one on the $y$-axis.
	
	In order to prove the first claim, we suppose by contradiction that the polygonal segments $\alpha_n$ are all contained in $\mathcal{Q}_3$. That implies that they are all straight line segments and the slope of each $\alpha_i$ is exactly $s_i$ from Lemma \ref{lem:lozi_recurr_coef}. Since
	\begin{equation*}
	\genfrac(){0pt}{0}{\lambda_Y^s}{b} = \genfrac(){0pt}{0}{\frac{1}{2}\left(a-\sqrt{a^2+4b}\right)}{b} = \tfrac{1}{2}\left(a-\sqrt{a^2+4b}\right)\genfrac(){0pt}{0}{1}{\frac{1}{2}\left(-a-\sqrt{a^2+4b}\right)},
	\end{equation*}
	 it follows that the limit of their slopes (Lemma \ref{lem:lozi_recurr_coef}(4)), $M_1=\frac{1}{2}\left(-a-\sqrt{a^2+4b}\right)$, corresponds to the stable direction of $L_{a,b}$ at its fixed point $Y$ in $\mathcal{Q}_3$. Combining this together with the fact that $L_{a,b}^{-1}$ is continuous, we see that there exist an $n'\in\mathbb{N}_0$ and $\eta>0$ such that
	 \begin{equation*}
	 \frac{\text{length}(\alpha_{n+1})}{\text{length}(\alpha_n)}>\eta>1,
	 \end{equation*}
	 for all $n\geqslant n'$, since $L_{a,b}^{-1}$ stretches all vectors in the direction $\genfrac(){0pt}{1}{\lambda_Y^s}{b}$ by the factor $\frac{1}{|\lambda_Y^s|}>1$ when it acts as an af\mbox{}fine map. The lengths of $\alpha_n$ are thus unboundedly increasing so one of those line segments will eventually intersect the $x$-axis, which is a contradiction with the original assumption that all segments $\alpha_n$ are contained in $\mathcal{Q}_3$. As a consequence, there exists an $n\in\mathbb{N}$ for which $\alpha_n$ intersects the $x$-axis. The smallest $n$ for which that happens yields the desired $n_0$, which proves the first claim of the lemma. 
	 
	From previous discussions it also immediately follows that $[V,V^{-n_0+1}]^{s}$ is contained in $\mathcal{Q}_3$, together with the claim about the  V-points of $W_X^{s}$ lying on it, since $\alpha_1,\alpha_2,\ldots,\alpha_{n_0}$ are all straight line segments.
		
	We now claim that $s_{n_0}<0$. Notice that $V^{-n_0-1}$ lies in the preimage of $\mathcal{Q}_2$, i.e.,\ parts of $\mathcal{Q}_1$ and $\mathcal{Q}_4$ below the curve $y=a|x|-1$, and $[V^{-n_0},V^{-n_0-1}]^{s}$ also has a V-point $P$ on the $y$-axis below $V$. The point $P$ is below the point $(0,-1)$ because $\alpha_{n_0}$ intersects the negative $x$-axis, so we have $[V^{-n_0},V^{-n_0-1}]^{s}=\overline{V^{-n_0}P}^{s}\cup\overline{PV^{-n_0-1}}^{s}$. If $s_{n_0}$ were positive, $\overline{V^{-n_0}P}^{s}$ would intersect $[V,V^{-(n_0-1)}]^{s}$, which is not possible. We thus conclude that $s_{n_0}$ is negative.
	
	Let $i_0$ be the smallest element of $\{1,\ldots,n_0\}$ such that $s_{i_0}<0$. Lemma \ref{lem:lozi_recurr_coef}(3) implies that $i_0$ is odd. Let $T_1$ be the intersection of the $x$-axis and the straight line through $V^{-i_0+2}$ and $V^{-i_0+1}$; see Figure \ref{figure:stable_zigzag}. Because $V^{-i_0+2}$ lies on $\overline{V^{-i_0+1}T_1}$, the points $V^{-i_0+1}$ will lie on $\overline{V^{-i_0}T_1^{-1}}$, i.e.,\ $V_{y}^{-i_0}>V_{y}^{-i_0+1}$. Similarly, if $T_2$ is the intersection of the $x$-axis and the straight line through $V^{-i_0}$ and $V^{-i_0-1}$, then $V^{-i_0-1}$ lies on $\overline{V^{-i_0}T_2^{-1}}$ so $V_{y}^{-i_0}>V_{y}^{-i_0-1}$. 
	
	Since Lemma \ref{lem:lozi_recurr_coef}(2) implies that $s_i<0$ for all $i\geqslant i_0$, we conclude by using the same argument that $V_{y}^{-n}>V_{y}^{-n+1}$ for all odd $n$ such that $i_0\leqslant n\leqslant n_0$. Namely, for any such $n$, let $T'$ be the intersection of the $x$-axis and the straight line through $V^{-n+2}$ and $V^{-n+1}$. Since $V^{-n+1}$ lies on $\overline{V^{-n+2}T'}$, then $V^{-n}$ will lie on $\overline{V^{-n+1}T'^{-1}}$ and thus $V_{y}^{-n}>V_{y}^{-(n-1)}$ because the slope of the straight line through $V^{-n+1}$ and $T'^{-1}$, $s_n$, is negative.
	
	Since $V_{y}^{-n_0}>0>V_{y}^{-n_0+1}$, it follows that $n_0$ is odd, which finishes the proof.
	\end{proof} 
	
	\begin{df}
	The broken line $[V,V^{-n_0}]^{s}$ from Lemma \ref{lem:zigzag} will be called the \emph{zigzag part} of $W_X^{s}$.
	\end{df}
	
	For example, on Figure \ref{fig:Lozi_stable_unstable_Z_V}, one can see that $n_0=5$, that is, the zigzag part of $W_X^s$ is $[V,V^{-5}]^{s}$.

	\subsection{Tangential homoclinic points}\label{subsec:border_hom_pts}
	
	Let $\eta$ and $\xi$ be two polygonal lines in the plane intersecting at a point $T$, and assume that $T$ is isolated from other intersection points of $\eta$ and $\xi$ in their respective topologies. For $\varepsilon>0$, let $\eta_{T,\varepsilon}$, $\xi_{T,\varepsilon}$ be the connected components of $\eta\cap B_{\varepsilon}(T)$, $\xi\cap B_{\varepsilon}(T)$ containing $T$, respectively.  
	Recall that there exist two possible types of intersections at $T$, as shown in Figure \ref{figure:tangential_transversal_intersec}. Firstly, $\eta$ and $\xi$ can intersect \emph{tangentially}; in that case, there exists an $\varepsilon>0$ such that $B_{\varepsilon}(T)\setminus \eta_{T,\varepsilon}$ consists of two connected components only one of which contains $\xi_{T,\varepsilon}$. Secondly, if such $\varepsilon$ does not exist, we say that the intersection at $T$ is \emph{transverse}. In addition, $\eta$ and $\xi$ can also \emph{intersect along a segment}, and this case will be dealt with in Subsection \ref{subsec:hom_intersec_segment}.
	
	\begin{figure}[!ht]
	\begin{center}
	\begin{tikzpicture}[auto, scale=1]
	\tikzstyle{nodec}=[draw,circle,fill=black,minimum size=2pt,
			inner sep=0pt, label distance=2mm]
			\tikzstyle{nodeh}=[draw,circle,fill=white,minimum size=4pt,
			inner sep=0pt]
			\tikzstyle{dot}=[circle,draw=none,fill=none,minimum size=0pt,inner sep=2pt, outer sep=-1pt]

			\begin{scope}
			\coordinate (t) at (0,0);			
			
			\coordinate (e1) at (0,1.25);
			\coordinate (e2) at (0,-0.84);
			\coordinate (e3) at (-1.45,-1.5);
			
			\draw[green!40!black, thick] (e1) node[left]{$\eta$} --(t)--(e2)--(e3) node[left]{$\eta$};
			
			\coordinate (k1) at (1,2);
			\coordinate (k2) at (1.6,1.2);
			\coordinate (k3) at (1,-1);
			
			\draw[violet!60!white, thick] (k1) node[left]{$\xi$}--(k2)--(t)--(k3) node[right]{$\xi$};	
			
			\node[nodec, label={[left]$T$}] at (t) {};
			\draw[thick, densely dotted] (t) circle (0.6);
			\end{scope}
			
			\begin{scope}[xshift=0.3333\linewidth, yshift=0cm]
			\coordinate (t) at (0,0);			
			
			\coordinate (e1) at (-2,1);
			\coordinate (e2) at (-1,2);
			\coordinate (e3) at (0.7,-1.4);
			
			\draw[green!40!black, thick] (e1) node[left]{$\eta$} --(e2)--(t)--(e3) node[right]{$\eta$};
			
			\coordinate (k1) at (0.5,1);
			\coordinate (k2) at (-1,-2);
			\coordinate (k3) at (-1.5,-1);
			\coordinate (k4) at (-2.2,-1.4);
			
			\draw[violet!60!white, thick] (k1) node[right]{$\xi$}--(t)--(k2)--(k3)--(k4) node[left]{$\xi$};	
			
			\node[nodec, label={[left]$T$}] at (t) {};
			\draw[thick, densely dotted] (t) circle (0.6);
			\end{scope}
			
			\begin{scope}[xshift=0.6666\linewidth, yshift=0cm]
			\coordinate (t1) at (0,0.75);
			\coordinate (t2) at (0,-0.75);			
			
			\coordinate (e1) at (1.5,2);
			\coordinate (e2) at (0.4,1.6);
			\coordinate (e3) at (-2,-1.2);
			
			\draw[green!40!black, thick] (e1) node[right]{$\eta$} --(e2)--(t1)--(t2)--(e3) node[left]{$\eta$};
			
			\coordinate (k1) at (-1.9,0.8);
			\coordinate (k2) at (-2,1.6);
			\coordinate (k3) at (0,-1.7);
			\coordinate (k4) at (-1.5,-2);
			
			\draw[violet!60!white, thick] (k1) node[left]{$\xi$}--(k2)--(t1)--(t2)--(k3)--(k4) node[left]{$\xi$};	
			
			\draw [thick, densely dotted] (0.6,0.75) arc (0:180:0.6);
			\draw [thick, densely dotted] (-0.6,-0.75) arc (180:360:0.6);
			
			\draw [thick, densely dotted] (-0.6,0.75)--(-0.6,-0.75);
			\draw [thick, densely dotted] (0.6,0.75)--(0.6,-0.75);
			
			\node [nodec, label={[right]$T_1$}] at (t1) {};
			\node [nodec, label={[right]$T_2$}] at (t2) {};
			\end{scope}
	\end{tikzpicture}
	\end{center}
	\caption{Intersections of two polygonal lines $\eta$ and $\xi$ in the plane: tangential (left), transverse (middle), intersection along a segment (right).}
	\label{figure:tangential_transversal_intersec}
	\end{figure}

	\begin{lem}\label{lemma:hom_tang}
	On the boundary of existence of homoclinic points for $X$, no intersection point of $W_X^u$ and $W_X^s$ dif\mbox{}ferent from $X$ is transverse, that is, $W_X^u$ and $W_X^s$ intersect tangentially or along a segment.
	\end{lem}
	
	\begin{proof}
	We know that the boundary of existence of homoclinic points for $X$ consists of parameter pairs $(a,b)$ in each open neighborhood of which there are corresponding Lozi maps which do and do not exhibit homoclinic points. Notice that all post-critical points on $W_X^s$ and V-points on $W_X^u$, and thus $W_X^s$ and $W_X^u$ themselves, depend continuously on $(a,b)$. Therefore, if $L_{a',b'}$ has a transverse homoclinic point for some $(a',b')$, that intersection will be transverse on some open neighborhood of $(a',b')$ in the parameter space, which implies that this point lies in the interior of the set of existence of homoclinic points for $X$. Hence, on the boundary of that set, no homoclinic intersections are transverse. 
	\end{proof}
	
	We will first consider the case when all homoclinic intersections are tangential. Each such homoclinic point is a V-point on $W_X^s$ or a post-critical point on $W_X^u$.
	
	In this subsection, we will consider the case when all intersections of $W_X^s$ and $W_X^u$ are tangential, the case of intersections along a segment will be studied in the next subsection. We will first describe the general structure of the unstable manifold $W_X^{u}$.
	
	Observe the forward orbit $(Z^n)_{n\in\mathbb{N}_0}$ of $Z$. Because $Z$ lies on the positive $x$-axis, from (\ref{eq:Z}) it follows that
	\begin{equation*}
	Z^1=\left(\frac{2-a-\sqrt{a^2+4b}}{2+a-\sqrt{a^2+4b}},\: \frac{2b}{2+a-\sqrt{a^2+4b}}\right).
	\end{equation*}
	Notice that for $a+b>1$ and $b<1$, we have $a^2>1-2b+b^2$, which implies
	\begin{multline*}
	a^2+a\sqrt{a^2+4b}+2b > 1-2b+b^2+a\sqrt{1-2b+b^2+4b}+2b \\
	 = 1+b(a+b)+a > 1+a+b = 2.
	\end{multline*}
	Hence,
	\begin{equation*}
	0>\frac{2-a^2-a\sqrt{a^2+4b}-2b}{2(1+a-b)}=\frac{2-a-\sqrt{a^2+4b}}{2+a-\sqrt{a^2+4b}}=Z_{x}^1,
	\end{equation*}	  
	and $Z^1$ lies in the second quadrant $\mathcal{Q}_2$. That is why $Z^2$ will lie in $\mathcal{Q}_4$ or $\mathcal{Q}_3$, and $[Z,Z^2]^{u}$ will not contain $Z^1$. In general, because $L_{a,b}$ is order reversing, $[Z,Z^{2k}]^{u}$ will not contain any odd iterates of $Z$ (for any given $k\in\mathbb{N}$) and vice versa; see Figure \ref{figure:unstable_structure}. For every $n \in \mathbb{N}_0$, we will put
	\begin{equation}
	\gamma_n = [Z^{2n-1},Z^{2n+1}]^{u},\quad \delta_n = [Z^{2n},Z^{2n+2}]^{u},
	\label{eq:gamma_n_delta_n}
	\end{equation}	 
and additionally,
	\begin{equation}
	\Gamma = \bigcup_{n=0}^{\infty}\gamma_n = W_X^{u-}\setminus[X,Z^{-1})^{u}, \quad \Delta = \bigcup_{n=0}^{\infty}\delta_n = W_X^{u+}\setminus[X,Z)^{u}.
	\label{eq:Gamma_Delta}
	\end{equation}
	
	\begin{figure}[!ht]
	\begin{center}
	\begin{tikzpicture}[auto]
			\tikzstyle{nodec}=[draw,circle,fill=black,minimum size=2pt,
			inner sep=0pt, label distance=2mm]
			\tikzstyle{nodeh}=[draw,circle,fill=white,minimum size=4pt,
			inner sep=0pt]
			\tikzstyle{dot}=[circle,draw=none,fill=none,minimum size=0pt,inner sep=2pt, outer sep=-1pt]

			\draw[->] (-5,0)--(5,0) node [below]{$x$};
			\draw[->] (0,-4)--(0,4) node [left]{$y$};
			\node[label={[xshift=-0.2cm, yshift=-0.7cm]$0$}] at (0,0) {};
			
			\coordinate (e0) at (0,0);
			\coordinate (ex) at (1,0);
			\coordinate (ey) at (0,1);
			
			\node[nodec, color=blue, label={[below right, color=blue]$Z$}] (t0) at (4,0) {};
			\node[nodec, color=blue, label={[above left, color=blue]$Z^1$}] (t1) at (-2,3) {};
			\coordinate (tm1) at (intersection of t0--t1 and e0--ey);
			\node[nodec, color=blue, label={[above right, color=blue]$Z^{-1}$}] at (tm1) {};
			\node[nodec, color=blue, label={[left, color=blue]$Z^3$}] (t3) at (-3,2) {};
			\node[nodec, color=blue, label={[below, color=blue]$Z^5$}] (t5) at (-2.5,-1) {};
			\node[nodec, color=blue, label={[below, color=blue]$Z^2$}] (t2) at (2.5,-3.5) {};
			\node[nodec, color=blue, label={[below, color=blue]$Z^4$}] (t4) at (-1,-3) {};
			\node[nodec, color=blue, label={[left, color=blue]$Z^6$}] (t6) at (-0.5,-1) {};
			
			\draw[blue, thick] (t0)--(tm1);
			\draw[blue, thick] (tm1) to node[dot, swap]{$\gamma_0$} (t1);
			\draw[blue, thick] (t1) to node[dot, swap]{$\gamma_1$} (t3);
			\draw[blue, thick] (t3) to node[dot, swap]{$\gamma_2$} (-1.8,0);
			\draw[blue, thick] (-1.8,0)--(t5);
			\draw[blue, thick] (t0) to node[dot, swap]{$\delta_0$} (t2);
			\draw[blue, thick] (t2) to node[dot, swap]{$\delta_1$} (t4);
			\draw[blue, thick] (t4)--(0.7,-2);
			\draw[blue, thick] (0.7,-2) to node[dot, swap]{$\delta_2$} (t6);
			\end{tikzpicture}
	\end{center}
	\caption{Sketch of the general structure of $W_X^u$. Here, we denote $\gamma_n = [Z^{2n-1},Z^{2n+1}]^u$ and $\delta_n = [Z^{2n},Z^{2n+2}]^u$ for all $n \in \mathbb{N}_0$.}
	\label{figure:unstable_structure}
	\end{figure}
	
	Recall that $W_X^{s+}$ denotes the part of $W_X^s$ that is a half-line starting at $X$ and going up in $\mathcal{Q}_1$. The following lemma shows that all homoclinic points for $X$ are always contained in the other part of $W_X^s$, that is, in $W_X^{s-}$.
	
	\begin{lem} \label{lem:Ws+X_WuX_intersect_at_X_only}
	$W_X^{u}\cap W_X^{s+}=\{X\}.$
	\end{lem}
	
	\begin{proof}
	Seeking a contradiction, assume that $W_X^u$ and $W_X^{s+}$ intersect at a point $A$ dif\mbox{}ferent from $X$. Since $W_X^{s+}$ is an $L_{a,b}$- and $L_{a,b}^{-1}$-invariant half-line in the upper half-plane, $L_{a,b}^{-1}$ acts on all points of $W_X^{s+}$ as an affine map. Therefore, for the Euclidean distances between backward iterates of $A$ and $X$ we have
	\begin{equation*}
	\dist(A^{-n},X)=\frac{1}{|\lambda_X^s|^n}\dist(A,X)\overset{n \rightarrow \infty}{\longrightarrow} \infty,
	\end{equation*}
	where $\lambda_X^s$ is the stable eigenvalue of $DL_{a,b}$ at $X$.
	
	On the other hand, since $A\in W_X^u$, we have by definition that $\dist(A^{-n},X)\overset{n \rightarrow \infty}{\longrightarrow}0$, which is a contradiction. Therefore, $W_X^u$ and $W_X^{s+}$ intersect at $X$ only.  
	\end{proof}
	
	\begin{lem}\label{lemma:hom1}
	Assume there exists a homoclinic point for the fixed point $X$. Then there exists a homoclinic point for $X$ on the line segment $\overline{XV}^{s}$.
	\end{lem}
	
	\begin{proof}
	If $T$, $T\neq X$, is a homoclinic point for $X$, then, because $T\in W^{s-}_{X}=\bigcup_{n=0}^{\infty}L_{a,b}^{-n}\bigl(\overline{XV}^s\bigr)$, it follows that there exists an $n\in\mathbb{N}_0$ such that $T\in L_{a,b}^{-n}\bigl(\overline{XV}^{s}\bigr)$. Therefore, $T^n \in \overline{XV}^{s}$, and $T^n \neq X$ since $L_{a,b}$ is a homeomorphism. Since $W^{u}_{X}$ is $L_{a,b}^n$-invariant, $T^n$ is the desired homoclinic point.
	\end{proof}
	
	In general, if $Z^{2n+1}$ is in $\mathcal{Q}_2$ above the curve $y=a|x|-1$ for some $n\in\mathbb{N}_0$, then $Z^{2n+2}$ is in $\mathcal{Q}_4$, and $Z^{2n+3}$ is in $\mathcal{Q}_2$ or $\mathcal{Q}_1$. Therefore, there exist parameter pairs for which neither $\Gamma$ nor $\Delta$ intersect the $y$-axis except at $Z^{-1}$, which lies on $\Gamma$. However, in the case when homoclinic points for $X$ exist, notice that $\Gamma$ intersects the $y$-axis at additional points because, due to Lemma \ref{lemma:hom1}, it intersects $W_X^s$ at points lying on $\overline{XV}^{s}$ in $\mathcal{Q}_1$ or $\mathcal{Q}_4$. This leads us to our two main cases of interest: either $\Delta$ intersects the $y$-axis, or it does not.
	
	\begin{figure}[!ht]
	\begin{center}
	\begin{tikzpicture}[auto]
			\tikzstyle{nodec}=[draw,circle,fill=black,minimum size=2pt,
			inner sep=0pt, label distance=2mm]
			\tikzstyle{nodeh}=[draw,circle,fill=white,minimum size=4pt,
			inner sep=0pt]
			\tikzstyle{dot}=[circle,draw=none,fill=none,minimum size=0pt,inner sep=2pt, outer sep=-1pt]

			\draw[->] (-5,0)--(5.5,0) node [below]{$x$};
			\draw[->] (0,-5)--(0,5) node [left]{$y$};
			\node[label={[xshift=-0.2cm, yshift=-0.7cm]$0$}] at (0,0) {};
			
			\coordinate (e0) at (0,0);
			\coordinate (ex) at (1,0);
			\coordinate (ey) at (0,1);
			
			\node[nodec, color=blue, label={[below right, xshift=-1mm, color=blue]$Z$}] (t0) at (4.7,0) {};
			\node[nodec, color=blue, label={[above right, color=blue]$Z^1$}] (t1) at (-3,4) {};
			\coordinate (tm1) at (intersection of t0--t1 and e0--ey);
			\node[nodec, color=blue, label={[above right, color=blue]$Z^{-1}$}] at (tm1) {};
			\node[nodec, color=blue, label={[left, color=blue]$Z^3$}] (t3) at (-3.5,3.5) {};
			
			\node[nodec, color=blue, label={[left, color=blue]$Z^{2i-3}$}] (t2im3) at (-3,2) {};
			\node[nodec, color=blue, label={[below, color=blue]$Z^{2i-1}$}] (t2im1) at (-1.5,0.5) {};
			\node[nodec, color=blue, label={[below right, xshift=-1mm, yshift=-1mm, color=blue]$Z^{2i+1}$}] (t2ip1) at (0.5,1) {};
			\node[nodec, color=blue, label={[right, xshift=-1mm, yshift=-2mm, color=blue]$Z^{2i+3}$}] (t2ip3) at (0.75,1.5) {};
			
			\node[nodec, color=blue, label={[right, color=blue]$Z^2$}] (t2) at (4.1,-1.5) {};
			\node[nodec, color=blue, label={[below, color=blue]$Z^4$}] (t4) at (3,-2.5) {};
			
			\node[nodec, color=blue, label={[below left, color=blue]$Z^{2i-2}$}] (t2im2) at (1.5,-2.5) {};
			\node[nodec, color=blue, label={[above left, xshift=2mm, yshift=-1mm, color=blue]$Z^{2i}$}] (t2i) at (1,-1) {};
			
			\coordinate (q) at (intersection of t2im1--t2ip1 and e0--ey);
			\node[nodec, label={[below, xshift=2mm]$A$}] at (q) {};
			
			\draw[blue, thick] (t0)--(t1);
			\draw[blue, thick] (t1)--(t3);
			\draw[blue, thick, dashed] (t3)--(t2im3);
			\draw[blue, thick] (t2im3) to node[dot, swap]{$\gamma_{i-1}$} (t2im1);
			\draw[blue, thick] (t2im1) to node[dot, swap]{$\gamma_i$} (t2ip1);
			\draw[blue, thick] (t2ip1)--(0.25,1.4)--(t2ip3);
			
			\draw[blue, thick] (t0)--(t2)--(t4);
			\draw[blue, thick, dashed] (t4)--(t2im2);
			\draw[blue, thick] (t2im2) to node[dot]{$\delta_{i-1}$} (t2i);
			
			\coordinate (d1) at (-3.5,4.4);
			\coordinate (d2) at (-3.5,-4.4);
			
			\coordinate (d) at (intersection of d1--t1 and e0--ex);
			\draw[green!40!black, thick, dashed] (d1)--(d)--(d2) node[below]{$x=1-\frac{a}{b}|y|$};
			
			\draw[thick, dashed] (4,1.5) node[right]{$y=a|x|-1$}--(0,-2.25)--(-4,1.5);
			
			\coordinate (f1) at (0,-2.25);
			\coordinate (f2) at (4,1.5);
			\coordinate (f3) at (-4,1.5);
			\coordinate (f) at (intersection of f1--f2 and e0--ex);
			\draw[red, thick] (f1) to node[dot, swap]{\scriptsize{$\varphi$}} (f);
			
			\draw[red, thick] (-3,0.5625) to node[dot, label={[above, color=red]\scriptsize{$L_{a,b}^{-1}(\varphi)$}}]{} (0,1.4);
			
			\begin{pgfonlayer}{bg}
				\fill [cyan!10!white] (q.center)--(t2im1.center)--(t2im3.center)--(t3.center)--(t1.center)--(tm1.center)--cycle;
			\end{pgfonlayer}
			
			\node[dot,label=$\mathcal{F}$] (ff) at (-2.5,2.5) {}; 
			\end{tikzpicture}
	\end{center}
	\caption{The figure illustrates the proof of Lemma \ref{lemma:unstable_yaxis_first}. Further iterations of $\gamma_i$ under $L_{a,b}^2$ do not intersect $L_{a,b}^{-1}(\varphi)$ outside the shaded polygon $\mathcal{F}$.}
	\label{figure:unstable_yaxis_first}
	\end{figure}
	
	\begin{lem}\label{lemma:unstable_yaxis_first}
	Assume that $\Delta$ intersects the $y$-axis. If $i_0$ is the smallest $i\in\mathbb{N}_0$ such that $\delta_i$ intersects that axis, then $\delta_{i_0}$ is a straight line segment, $\delta_{i_0}=\overline{Z^{2i_0}Z^{2i_0+2}}^{u}$.
	\end{lem} 
	
	\begin{proof}
	We claim that there are no other post-critical points on $\delta_{i_0}$ apart from the corresponding iterates of $Z$. Assume, for sake of contradiction, that the converse holds. Then $\gamma_i$ transversely intersects the $y$-axis for some $i\in\mathbb{N}_0$, $0<i\leqslant i_0$. Observe the broken line $[Z^1,Z^{2i+1}]^{u}$. Because of the choice of $i_0$, that line does not intersect the curve $y=a|x|-1$, and thus it does not intersect the $x$-axis either. Therefore, the straight line segment $\gamma_i$ will intersect the positive $y$-axis at some point $A$. Let $\mathcal{F}$ denote the polygon in the second quadrant $\mathcal{Q}_2$ with boundary $\partial\mathcal{F}=[Z^{-1},A]^{u}\cup\overline{AZ^{-1}}$. Observe that, because $Z_{x}>1$, the segment $\overline{Z^{-1}Z^1}^{u}$ is above, and the polygonal line $[Z^1,A]^{u}$ is below the curve $x=1-\frac{a}{b}|y|$ in the upper half-plane, since $[Z^1,A]^{u}$ is obtained from images of the corresponding $\delta_n$ which lie in $\mathcal{Q}_4$; see Figure \ref{figure:unstable_yaxis_first}.
	
	Notice that if $\gamma_i$ intersects the positive $y$-axis, then $\delta_{i-1}$ intersects the portion of the curve $y=a|x|-1$ in $\mathcal{Q}_4$. That portion is a straight line segment which will be denoted by $\varphi$. In that case, $\gamma_{i-1}$ intersects $L_{a,b}^{-1}(\varphi)$, which is a line segment in the left half-plane, with one endpoint lying on $y=a|x|-1$ and the other one on the $y$-axis. In general, if $\gamma_l$ intersects the $y$-axis below $A$ for some $l\in\mathbb{N}$, then $\gamma_{l-1}$ intersects $L_{a,b}^{-1}(\varphi)$ outside $\mathcal{F}$.
	
	On the other hand, observe the forward images of $\gamma_i$ under $L_{a,b}^2$. We know that points in $\mathcal{Q}_1$ are mapped under $L_{a,b}$ to $\mathcal{Q}_1$ or $\mathcal{Q}_2$ above the curve $x=1-\frac{a}{b}|y|$, and points in $\mathcal{Q}_2$ which lie to the right of the curve $y=a|x|-1$ are mapped to those in $\mathcal{Q}_4$. Therefore, in order for some $\gamma_{j}$ to reach $\mathcal{Q}_2$, it will first intersect $\overline{AZ^{-1}}$, i.e., the polygon $\mathcal{F}$. Since $W_X^{u}$ has no self-intersections, further forward images of $\gamma_j$ can only intersect $L_{a,b}^{-1}(\varphi)$ inside $\mathcal{F}$, which leads us to the conclusion that no further forward images of $\gamma_i$ under $L_{a,b}^2$ intersect the $y$-axis below $A$. By consequence, no further forward images of $\gamma_i$ under $L_{a,b}$ intersect the curve $y=a|x|-1$ in $\mathcal{Q}_2$ either. This is a contradiction with the assumption that $\Delta$ intersects the negative $y$-axis, so $\delta_{i_0}$ is indeed a straight line segment.
	\end{proof}
	
	We want to determine tangential homoclinic points for $X$ in the border case of their existence. In order to do that, we need to determine which post-critical points of $W_X^u$ lie on $W_X^s$, as well as which V-points of $W_X^s$ lie on $W_X^u$.
	
	Notice that all post-critical points on $W_X^u$ are forward iterates of points at which $W_X^u$ transversely intersects the $y$-axis. In the following two lemmas, we show that these intersection points are contained in polygons bounded by parts of $W_X^u$ and parts of $W_X^s$.
	
	Let $\gamma_n,\delta_n\subset W_X^u$ and $\Gamma,\Delta\subset W_X^u$ be as in (\ref{eq:gamma_n_delta_n}) and (\ref{eq:Gamma_Delta}), respectively; see Figure \ref{figure:unstable_structure}. Also, let $[V,V^{-n_0}]^s$ be the zigzag part of $W_X^s$, as introduced in Lemma \ref{lem:zigzag}.
	
	\begin{lem} \label{lem:main_thm_lemma1}
	Assume that all intersections of $W_X^u$ and $W_X^s$ are tangential except the one at $X$. Assume also that $\Delta$ does not intersect the $y$-axis. Then all intersections of $W_X^{u}$ with the $y$-axis are contained in the triangle $XVZ^1$.
	\end{lem}
	
	\begin{proof}	
	Lemma \ref{lemma:hom1} implies that there exists a homoclinic point point for $X$ on $\overline{XV}^s$. Moreover, since $\overline{XV}^s=\{X\}\cup\bigcup_{n=0}^{\infty}L_{a,b}^n\bigl(\overline{VV^1}^s\bigr)$, we see that there exists a homoclinic point for $X$ on $\overline{VV^1}^s$ and all of its forward and backward iterates as well. In particular, there exists a homoclinic point for $X$ on $\overline{V^{-n_0+1}V^{-n_0}}^s$, where $n_0\in\mathbb{N}$ is such that $[V,V^{-n_0}]^{s}$ is the zigzag part of $W_X^s$.
	
	We now claim that $n_0=1$. Assume by contradiction that the converse holds, $n_0>1$. Lemma \ref{lem:zigzag} then implies that $n_0\geqslant3$. In that case, $L_{a,b}\bigl(\overline{V^{-n_0+1}V^{-n_0}}^{s}\bigr)$ and $L_{a,b}^2\bigl(\overline{V^{-n_0+1}V^{-n_0}}^{s}\bigr)$ are two line segments that both belong to the zigzag part of $W_X^s$, so they are both contained in the third quadrant $\mathcal{Q}_3$ and there is a homoclinic point on each one of them. Notice that this implies that a homoclinic point on one of those segments belongs to $\Delta$: if a homoclinic point on $L_{a,b}^2\bigl(\overline{V^{-n_0+1}V^{-n_0}}^{s}\bigr)$ lies on $\Gamma$, then the preimage of that point lies on $\Delta$. Therefore, $\Delta$ intersects $\mathcal{Q}_3$. On the other hand, the assumption that $\Delta$ does not intersect the $y$-axis implies that $\Delta\subset\mathcal{Q}_1\cup\mathcal{Q}_4$. This yields a contradiction, so it follows that $n_0=1$.
	
	Furthermore, $V^{-1}$ lies on the curve $y=a|x|-1$ and because $V_{y}>-1$, the segment $\overline{VV^{-1}}^{s}$ lies above that curve in $\mathcal{Q}_2\cup\mathcal{Q}_3$. However, because portions of further backward images of $\overline{VV^{-1}}^{s}$ in the $\mathcal{Q}_3$ have negative slopes (Lemma \ref{lem:zigzag}, Lemma \ref{lem:lozi_recurr_coef}(2)), and they intersect the negative $y$-axis, they will all be below the curve $y=a|x|-1$ in $\mathcal{Q}_3$ and intersect it in $\mathcal{Q}_2$. 
	
	In addition, since $W_X^{u}=\bigcup_{n=0}^{\infty}L_{a,b}^n\bigl(\overline{XZ}^{u}\bigr)$ and $\overline{XZ}^{u}=\{X\}\cup\bigcup_{n=0}^{\infty}L_{a,b}^{-2n}\bigl(\overline{ZZ^{-2}}^{u}\bigr)$, it follows that if there are homoclinic points for $X$, then there is also a homoclinic point for $X$ lying on $\overline{ZZ^{-2}}^{u}$ and all of its forward and backward iterates. In particular, there exists a homoclinic point for $X$ on $\gamma_n$ and $\delta_n$ for all $n\in\mathbb{N}_0$.
	
		\begin{figure}[!ht]
	\begin{center}
	\begin{tikzpicture}[auto, xscale=1.5]
			\tikzstyle{nodec}=[draw,circle,fill=black,minimum size=2pt,
			inner sep=0pt, label distance=2mm]
			\tikzstyle{nodeh}=[draw,circle,fill=white,minimum size=4pt,
			inner sep=0pt]
			\tikzstyle{dot}=[circle,draw=none,fill=none,minimum size=0pt,inner sep=2pt, outer sep=-1pt]

			\draw[->] (-3,0)--(4,0) node [below]{$x$};
			\draw[->] (0,-4)--(0,6) node [left]{$y$};
			\node[label={[xshift=-0.2cm, yshift=-0.7cm]$0$}] at (0,0) {};
			
			\coordinate (e0) at (0,0);
			\coordinate (ex) at (1,0);
			\coordinate (ey) at (0,1);
			
			\coordinate (x) at (0.9,0.925);

			\node[nodec, color=red, label={[below right, color=red]$V$}] (v0) at (0,-2) {};
			\node[nodec, color=red, label={[above right, color=red]$V^{-1}$}] (vm1) at (-2,4) {};
			\node[nodec, color=red, label={[above, color=red]$V^{-2}$}] (vm2) at (3,1) {};
			\coordinate (a) at (2,4.5);
			\draw[red, thick] (a)--(v0)--(vm1)--(0,-3)--(vm2)--(0,-3.5)--(-3,3);
			
			\draw[thick, dashed] (-2.5,5.625)--(0,-2.5)--(2.5,5.625) node[right]{$y=a|x|-1$};
			
			\coordinate (b) at (0,-3);
			\coordinate (t0) at (intersection of b--vm2 and e0--ex);
			
			\coordinate (t1) at (intersection of t0--x and v0--vm1);
			\coordinate (t2) at (0.2,-1.35);
			\coordinate (t3) at (0.7,0.275);
			
			\draw[blue, thick] (t3)--(t1)--(t0)--(t2);
			
			\draw[red, thick] (b) to node[dot]{$\theta$} (vm2);
			
			\node[nodec, color=blue, label={[above, color=blue]$Z$}] at (t0) {};
			\node[nodec, label={[above left, xshift=2mm]$X$}] at (x) {};
			\node[nodec, color=blue, label={[above right, color=blue]$Z^1$}] at (t1) {};
			\node[nodec, color=blue, label={[above left, xshift=3mm, color=blue]$Z^2$}] at (t2) {}; 
			\node[nodec, color=blue, label={[above, xshift=-0.5mm, yshift=0.5mm, color=blue]$Z^3$}] at (t3) {}; 
			
			\begin{pgfonlayer}{bg}
				\fill [magenta!10!white] (x.center)--(v0.center)--(t1.center)--cycle;
			\end{pgfonlayer}
			\end{tikzpicture}
	\end{center}
	\caption{The figure illustrates the proof of Lemma \ref{lem:main_thm_lemma1}: $Z$ lies on the line segment $\theta$ (portion of $[V^{-2},V^{-1}]^{s}$ in the right half-plane), $Z^2$ on $\overline{VV^1}^{s}$, and $\Gamma$, together with all of its intersections with the $y$-axis, is contained in the shaded triangle $XVZ^1$.}
	\label{figure:thm_hom_case1}
	\end{figure}
	
	Now observe $\delta_0$. Since $\gamma_0$ is a line segment in $\mathcal{Q}_2$, then $\delta_0$ will also be a line segment. Due to the assumption that $\Delta$ does not intersect the $y$-axis, $\delta_0$ lies in $\mathcal{Q}_4$. We know there exists a homoclinic point for $X$ lying on $\delta_0$, which implies that this line segment intersects $W_X^{s}$ in $\mathcal{Q}_4$. On the other hand, $W_X^{s}\cap\mathcal{Q}_4$ consists of $\overline{VV^1}^{s}$ and the segments $\nu_n:=L_{a,b}^{-1}\Bigl(L_{a,b}^{-n}\bigl(\overline{VV^{-1}}^{s}\bigr)\cap\mathcal{Q}_2\Bigr)\cap\mathcal{Q}_4$, for all $n\in\mathbb{N}$. 
	
	Assume that $\delta_0$ intersects one of the segments $\nu_n$ for some $n>1$. Then $\gamma_1$ intersects $L_{a,b}^{-n}\bigl(\overline{VV^{-1}}^{s}\bigr)$ in $\mathcal{Q}_2$ below the curve $y=a|x|-1$. Specially, this implies that $\Gamma$ intersects the curve $y=a|x|-1$. On the other hand, $\Delta$ does not intersect the $y$-axis, and therefore $\Gamma$ does not intersect its preimage $y=a|x|-1$. This is a contradiction, so it follows that $\delta_0$ intersects $\overline{VV^1}^{s}$ or $L_{a,b}^{-1}\bigl(\overline{VV^{-1}}^{s}\bigr)\cap(\mathcal{Q}_4\cup\mathcal{Q}_1)$, which is a line segment. We denote that line segment by $\theta$. 
	
	However, the endpoints of $\overline{VV^1}^{s}$ are $V^1$ and $V$ (which lies on the $y$-axis), and those of $\theta$ are $V^{-2}$ and a V-point on the $y$-axis. Since $\Delta$ does not intersect the $y$-axis, and the homoclinic intersection on $\delta_0$ is tangential, we conclude that every homoclinic point on $\delta_0$ is a post-critical point, that is, $Z$ and $Z^2$. In this case, we see that $Z$ lies on $\theta$, and $Z^2$ lies on $\overline{VV^1}^{s}$. Therefore, $Z^1$ lies on $\overline{VV^{-1}}^{s}$ in $\mathcal{Q}_2$, and $Z^3$ lies on $\overline{V^1V^2}^{s}$ in $\mathcal{Q}_1$, so $\gamma_1$ is the first element of the sequence $(\gamma_n)_{n\in\mathbb{N}}$ that intersects the $y$-axis.
	
	Finally, observe now the triangle $XVZ^1$ and its boundary, $[X,Z^1]^{s}\cup[Z^1,X]^{u}$. Notice that $\Gamma$ is contained in that triangle since it cannot transversely intersect its sides. Therefore, all intersections of $W_X^{u}$ with the $y$-axis are contained in that triangle, which finishes the proof.
	\end{proof}
	
	We now observe the case when $\Delta \subset W_X^u$ (given by (\ref{eq:Gamma_Delta})) intersects the $y$-axis. Recall that $\delta_{i_0}$ is the first element of the sequence $(\delta_n)_{n \in \mathbb{N}_0}$ that intersects the $y$-axis.
	
	\begin{lem} \label{lem:main_thm_lemma2}
	Assume that all intersections of $W_X^u$ and $W_X^s$ are tangential except the one at $X$. Also assume that $\Delta$ intersects the $y$-axis, and let $M$ be the first homoclinic point on $\Delta$ that lies in $\mathcal{Q}_3$, counting from $Z$. Let $\mathcal{G}$ be the polygon with boundary $\partial\mathcal{G}=[X,M]^u\cup[M,X]^s$. Then the folowing statements hold.
	\begin{enumerate}
		\item All intersections of $\Delta$ with the $y$-axis are contained in $\mathcal{G}$, and all intersections of $\Gamma$ with the $y$-axis are contained in $L_{a,b}(\mathcal{G})$.
		\item If there exists a homoclinic point $\delta_{i_0}$ that is a V-point on the $y$-axis, then that point is also an iterate of $Z$ (more precisely, $Z^{2i_0}$ or $Z^{2i_0+2}$).
	\end{enumerate}
	\end{lem}
	
	\begin{proof}
	Let $i_0$ be the smallest $i\in\mathbb{N}_0$ such that $\delta_i$ intersects the $y$-axis, as in Lemma \ref{lemma:unstable_yaxis_first}. Like before, let $[V,V^{-n_0}]^{s}$ be the zigzag part of $W_X^{s}$. We claim that $M$, the first homoclinic point that occurs on $\Delta$ in the third quadrant $\mathcal{Q}_3$ counting from $Z$, lies either on the zigzag part of $W_X^{s}$ or on $L_{a,b}^{-1}\bigl(\overline{V^{-n_0+1}V^{-n_0}}^{s}\bigr)$. 
	
	\begin{figure}[!ht]
	\begin{center}
	\begin{tikzpicture}[auto]
			\tikzstyle{nodec}=[draw,circle,fill=black,minimum size=2pt,
			inner sep=0pt, label distance=2mm]
			\tikzstyle{nodeh}=[draw,circle,fill=white,minimum size=4pt,
			inner sep=0pt]
			\tikzstyle{dot}=[circle,draw=none,fill=none,minimum size=0pt,inner sep=2pt, outer sep=-1pt]

			\draw[->] (-7,0)--(2,0) node [below]{$x$};
			\draw[->] (0,-7)--(0,1.5) node [left]{$y$};
			\node[label={[xshift=-0.2cm, yshift=-0.7cm]$0$}] at (0,0) {};
			
			\coordinate (e0) at (0,0);
			\coordinate (ex) at (1,0);
			\coordinate (ey) at (0,1);
			
			\coordinate (xi) at (-5,-6);
			\draw[green!40!black, thick, dashed] (ex)--(xi) node[below]{$x=1-\frac{a}{b}|y|$};
			
			\coordinate (a1) at (-1,-1);
			\coordinate (a2) at (-0.50291, -2.49484);
			\coordinate (a3) at (-4.11632, 0.66169);
			\coordinate (a4) at (0,-4.5);
			\coordinate (a5) at (0.77078, -3.92);
			\coordinate (a6) at (0,-5);
			\coordinate (a7) at (-4.6,0);
			\coordinate (a8) at (-5.2,0);
			\coordinate (a9) at (0,-5.45);
			\coordinate (a10) at (-5.5,0);
			\coordinate (a11) at (0,-5.82);
			\coordinate (a12) at (-6,0);
			\coordinate (a13) at (0,-6.4);
			
			\coordinate (b1) at (intersection of ex--xi and a1--a2);
			\coordinate (b2) at (intersection of ex--xi and a2--a3);
			\coordinate (b3) at (intersection of ex--xi and a3--a4);
			\coordinate (b4) at (intersection of ex--xi and a6--a7);
			\coordinate (b5) at (intersection of ex--xi and a8--a9);
			\coordinate (b6) at (intersection of ex--xi and a10--a11);
			\coordinate (b7) at (intersection of ex--xi and a12--a13);
			
			\draw[red, thick] (a1)--(b1);
			\draw[red, thick] (b2)--(a3)--(b3);
			\draw[red, thick] (a4)--(a5)--(a6);
			\draw[red, thick] (b4)--(a7);
			\draw[red, thick] (b5)--(a8);
			\draw[red, thick] (b6)--(a10);
			\draw[red, thick] (b7)--(a12);
			
			\draw[violet, thick] (b1) to node[dot]{$\beta_0$} (a2);
			\draw[violet, thick] (b2) to node[dot, swap]{$\beta_1$} (a2);
			\draw[violet, thick] (b3) to node[dot]{$\beta_2$} (a4);
			\draw[violet, thick] (b4)--(a6);
			\draw[violet, thick] (b5)--(a9);
			\draw[violet, thick] (b6)--(a11);
			\draw[violet, thick] (b7)--(a13);
			
			\node[nodec, color=red, label={[above, xshift=-1.5mm, color=red]$V^{-n_0+2}$}] at (a1) {};
			\node[nodec, color=red, label={[right, color=red]$V^{-n_0+1}$}] at (a2) {};			
			\node[nodec, color=red, label={[above, color=red]$V^{-n_0}$}] at (a3) {};
			\node[nodec, color=red, label={[above right, color=red]$V^{-n_0-1}$}] at (a5) {};
			\node[nodec, color=red, label={[above, xshift=2.3mm, yshift=1.5mm, color=red]$P_1$}] at (a4) {};
			\node[nodec, color=red, label={[below right, color=red]$P_2$}] at (a6) {};
					
			\end{tikzpicture}
	\end{center}
	\caption{Portions of $W_X^s$ in the third quadrant $\mathcal{Q}_3$: segments $\beta_n$ (violet) and $\alpha_n$. All possible homoclinic points on $\delta_{i_0}$ in the third quadrant can lie on $\beta_0$, $\beta_1$ or $\beta_2$ only.}
	\label{figure:thm_hom_stable_segments}
	\end{figure}
	
	Observe $\delta_{i_0}$. Lemma \ref{lemma:hom_tang} implies that all homoclinic points on $\delta_{i_0}$ are either post-critical points on $W_X^u$ or V-points on $W_X^s$. Since $\delta_{i_0}$ is a line segment by Lemma \ref{lemma:unstable_yaxis_first}, the only post-critical points on $\delta_{i_0}$ are $Z^{2i_0}$ and $Z^{2i_0+2}$. Assume that $Z$ is not a homoclinic point for $X$ and let us now focus on homoclinic points on $\delta_{i_0}$ that are V-points on $W_X^s$ (if such points exist).
	
	Assume first that all of these points lie in the fourth quadrant $\mathcal{Q}_4$. Notice that for all V-points $S$ in $\mathcal{Q}_4$, the second iterate $S^2$ lies again in $\mathcal{Q}_4$ unless $S$ is $V^{-n_0-1}$ (in that case, $V^{-n_0+1}$ lies in $\mathcal{Q}_3$). Assume for sake of contradiction that $S$ lies on $\delta_{i_0}$, and $S \neq V^{-n_0-1}$. Then $W_X^s$ intersects $\delta_{i_0+1}$ at $S^2 \in \mathcal{Q}_4$. Let $R$ be the intersection point of $\delta_{i_0}$ with the $y$-axis. Notice that in $\mathcal{Q}_4$, $\delta_{i_0+1}$ is above $[Z,R]^u$ since $\delta_{i_0+1}$ intersects the curve $x=1-\frac{a}{b}|y|$, and $[Z,R]^u$ lies below that curve. Therefore, to intersect $\delta_{i_0+1}$ in $\mathcal{Q}_4$, the stable manifold $W_X^s$ will transversely intersect $[Z,R]^u$. This yields a contradiction since the intersections of $W_X^u$ and $W_X^s$ apart from $X$ are tangential. Therefore, $S=V^{-n_0-1}$ is a homoclinic point on $\delta_{i_0}$. Then $M=V^{-n_0+1}$, and $V^{-n_0+1}$ lies on $\delta_{i_0+1}$ as a V-point on the zigzag part, which proves the claim in this case. 
	
	Now assume that there exist homoclinic points on $\delta_{i_0}$ that are V-points on $W_X^s$, and that also lie in $\mathcal{Q}_3$ or on the $y$-axis. For every $n\in\mathbb{N}_0$, let $\alpha_n := L_{a,b}^{-n}\bigl(\overline{V^{-n_0+2}V^{-n_0+1}}^{s}\bigr)\cap\mathcal{Q}_3$, and let $\beta_n$ be the portion of $\alpha_n$ below the curve $x=1-\frac{a}{b}|y|$; see Figure \ref{figure:thm_hom_stable_segments}. Observe that by the construction of the zigzag part of $W_X^s$, $\beta_n$ is non-empty for every $n\in\mathbb{N}_0$, and $\alpha_n$ is a straight line segment with endpoints lying on the coordinate axes for every $n\geqslant 2$. Moreover, notice that $L_{a,b}^{-1}(\alpha_n\setminus\beta_n)=\alpha_{n+1}$ for every $n\in\mathbb{N}_0$.
	
	Homoclinic points on $\delta_{i_0}$ in $\mathcal{Q}_3$ are forward images of the corresponding homoclinic point on $\gamma_{i_0-1}$ in $\mathcal{Q}_2$, that is, every homoclinic point on $\delta_{i_0}$ in $\mathcal{Q}_3$ also lies in $L_{a,b}(\mathcal{Q}_2)$. Therefore, every homoclinic point on $\delta_{i_0}$ in $\mathcal{Q}_3$ lies on some $\beta_k$. We claim that all such points lie on $\beta_0$, $\beta_1$ or $\beta_2$. Assume by contradiction that the converse holds, there is a homoclinic point on $\delta_{i_0}\cap\mathcal{Q}_3$ that lies on some $\beta_k$ for $k>2$. Then $\delta_{i_0+1} \subset W_X^u$ transversely intersects $\alpha_{k-1} \subset W_X^s$ in $\mathcal{Q}_3$ in order to intersect $\alpha_{k-2}$ and the curve $x=1-\frac{a}{b}|y|$. This is a contradiction with the assumption that all intersections of $W_X^s$ and $W_X^u$ apart from $X$ are tangential. It follows that $k\leqslant2$, all homoclinic points on $\delta_{i_0}\cap\mathcal{Q}_3$ lie on $\beta_0$, $\beta_1$ or $\beta_2$. This also completes the proof of the claim that the first homoclinic point on $\Delta$ in $\mathcal{Q}_3$, counting from $Z$, lies either on the zigzag part of $W_X^{s}$ or on $L_{a,b}^{-1}\bigl(\overline{V^{-n_0+1}V^{-n_0}}^{s}\bigr)$.
	
	We now prove claim (2), that is, if a homoclinic point on $\delta_{i_0}$ is a V-point on the $y$-axis, then that point is also an iterate of the point $Z$. Suppose that a homoclinic point $P_1$ on $\delta_{i_0}$ is a V-point on the $y$-axis. In the previous paragraph we proved that all homoclinic points on $\delta_{i_0}\cap\mathcal{Q}_3$ lie on the segments $\beta_0$, $\beta_1$ or $\beta_2$; see Figure \ref{figure:thm_hom_stable_segments}. It follows that $P_1$ lies on $\beta_2$, that is, $P_1$ is the preimage of $\overline{V^{-n_0+1}V^{-n_0}}^{s}\cap\{x\text{-axis}\}$. Let now $P_2$ be the preimage of $\overline{V^{-n_0}P_1}^{s}\cap\{x\text{-axis}\}$, and observe the triangle $P_1V^{-n_0-1}P_2$. Suppose by contradiction that $P_1$ is not a post-critical point on $\delta_{i_0}$. In order to tangentially intersect $W_X^s$ at $P_1$, the segment $\delta_{i_0}$ will transversely intersect one of the sides of the triangle $P_1V^{-n_0-1}P_2$ that does not lie on the $y$-axis, that is, one of the sides lying on the polygonal segment $[P_1,P_2]^{s}$. This is a contradiction with the assumption that all intersections of $W_X^u$ and $W_X^s$ apart from $X$ are tangential. Therefore, $P_1$ is a post-critical point on $\delta_{i_0}$ and thus an iterate of $Z$ ($Z^{2i_0}$ or $Z^{2i_0+2}$). This finishes the proof of claim (2).
		
	\begin{figure}[!ht]
	\begin{center}
	\begin{tikzpicture}[auto, scale=0.6]
			\tikzstyle{nodec}=[draw,circle,fill=black,minimum size=2pt,
			inner sep=0pt, label distance=2mm]
			\tikzstyle{nodeh}=[draw,circle,fill=white,minimum size=4pt,
			inner sep=0pt]
			\tikzstyle{dot}=[circle,draw=none,fill=none,minimum size=0pt,inner sep=2pt, outer sep=-1pt]

			\draw[->] (-11,0)--(9,0) node [below]{$x$};
			\draw[->] (0,-10.5)--(0,6.5) node [left]{$y$};
			\node[label={[xshift=-0.2cm, yshift=-0.7cm]$0$}] at (0,0) {};
			
			\coordinate (e0) at (0,0);
			\coordinate (ex) at (1,0);
			\coordinate (ey) at (0,1);
			
			\coordinate (s11) at (1.70611, 5.22029);
			\coordinate (s1) at (0,-3);
			\coordinate (s2) at (-1.37846, -1.49641);
			\coordinate (s3) at (-0.54099, -4.52417);
			\coordinate (s4) at (-3.34328, -0.88442);
			\coordinate (s5) at (-1.08857, -5.61932);
			\coordinate (s6) at (-8.91564, 1.49913);
			\coordinate (s7) at (0,-9.3);
			\coordinate (s8) at (1.3518, -8.56567);
			\coordinate (s9) at (0,-10);
			\coordinate (s10) at (-8.89105, 0.45318);
			
			\coordinate (u1) at (7.5,0);
			\coordinate (u2) at (6.95637, -3.18657);
			\coordinate (u3) at (5.57133, -5.31244);
			\coordinate (u4) at (1.86716, -7.30948);
			\coordinate (u5) at (-2.78, 5.8);
			\coordinate (u6) at (-5.56578, 4.78457);
			\coordinate (u7) at (-7.17629, 3.52838);
			\coordinate (u8) at (-7.28054, 1.22623);
			
			\coordinate (x) at (intersection of u1--u5 and s1--s11);
			
			\coordinate (p1) at (-2,-7.5);
			\coordinate (p2) at (-1,-7.5);
			\coordinate (q1) at (-2,-1.8);
			\coordinate (q2) at (-1,-1.8);
			
			\coordinate (m) at (intersection of s6--s7 and p1--p2);
			\coordinate (lm) at (intersection of s5--s6 and q1--q2);
			
			\draw[blue, thick] (u7)--(u6)--(u5)--(u1)--(u2)--(u3);
			\draw[blue, thick, dashed] (u7)--(u8);
			\draw[blue, thick, dashed] (u3)--(u4);
			\draw[blue, thick] (u8)--(lm);
			\draw[blue, thick] (u4) to node[dot]{$\delta_{i_0}$} (m);
			
			\draw[red, thick] (s11)--(s1)--(s2)--(s3)--(s4);
			\draw[red, thick, dashed] (s4)--(s5);
			\draw[red, thick] (s5)--(s6)--(s7)--(s8)--(s9)--(s10);	
			
			\node[nodec, color=red, label={[below right, color=red]$V$}] at (s1) {};
			\node[nodec, color=blue, label={[below right, color=blue]$Z$}] at (u1) {};
			\node[nodec, label={[above right]$X$}] at (x) {};
			\node[nodec, label={[above, xshift=2mm]$M$}] at (m) {};
			\node[nodec, label={[right, yshift=1mm]$M^1$}] at (lm) {};
			
			\node[dot, label=$\mathcal{G}$] (gg) at (5,-2) {};
			\node[dot, label=$L_{a,b}(\mathcal{G})$] (hh) at (-4,3.5) {};
			
			\begin{pgfonlayer}{bg}
				\fill[green!10!white] (x.center)--(u5.center)--(u6.center)--(u7.center)--(u8.center)--(lm.center)--(s5.center)--(s4.center)--(s3.center)--(s2.center)--(s1.center)--cycle;
				
				\fill[yellow!10!white] (x.center)--(u1.center)--(u2.center)--(u3.center)--(u4.center)--(m.center)--(s6.center)--(s5.center)--(s4.center)--(s3.center)--(s2.center)--(s1.center)--cycle;
			\end{pgfonlayer}	
			\end{tikzpicture}
	\end{center}
	\caption{Polygon $\mathcal{G}$ and its image $L_{a,b}(\mathcal{G})$, as in Lemma \ref{lem:main_thm_lemma2}. Point $M$ is the first homoclinic point on $\Delta$ in $\mathcal{Q}_3$, counting from $Z$.}
	\label{figure:thm_hom_case2}
	\end{figure}
	
	Finally, to prove claim (1), as before, let $M$ be the first homoclinic point on $\Delta$ lying in $\mathcal{Q}_3$, counting from $Z$. Let $\mathcal{G}$ be the polygon whose boundary is $\partial\mathcal{G}=[X,M]^{u}\cup[M,X]^{s}$; see Figure \ref{figure:thm_hom_case2}. Notice that the only V-points contained in $[M,X]^{s}$ are iterates of $V$ and possibly $M$ itself (the previous paragraph implies that in that case, $M$ coincides with some iterate of $Z$). Moreover, since $W_X^{u}$ does not transversely intersect the sides of $\mathcal{G}$ and $[M,M^2]^{u}$ is contained in that polygon, $\mathcal{G}$ contains $\Delta$ and all of its intersections with the $y$-axis. Similarly, $L_{a,b}(\mathcal{G})$ is a polygon with boundary $\partial L_{a,b}(\mathcal{G})=[X,M^1]^{u}\cup[M^1,X]^{s}$, and $L_{a,b}(\mathcal{G})$ contains $\Gamma$ and all of its intersections with the $y$-axis (notice that $[X,M^1]^{u}$ intersects the $y$-axis at $Z^{-1}$ only). This finishes the proof of the lemma.
	\end{proof}
	
	\begin{theorem} \label{thm:hom_tangential}
	Assume that all intersections of $W_X^s$ and $W_X^u$ are tangential except the one at $X$. Then the set of all homoclinic points for $X$ is one of the following: the orbit of $Z$, the orbit of $V$, or the union of these orbits.
	\end{theorem}
	
	\begin{proof}
	Lemma \ref{lemma:hom1} implies that it suf{}fices to observe all possible homoclinic points on $\overline{XV}^{s}$. Furthermore, because $\overline{XV^1}^{s}=\{X\}\cup\bigcup_{n=1}^{\infty}L_{a,b}^n\bigl(\overline{VV^1}^{s}\bigr)$, it is enough to observe homoclinic points on $\overline{VV^1}^{s}$. From Lemma \ref{lemma:hom_tang} it follows that two possibilities can occur in this case. Firstly, $V$ can be a homoclinic point as a V-point on $W_X^{s}$, and the claim of the theorem follows. Secondly, there is a post-critical point on $W_X^{u}$ lying on $\overline{VV^1}^{s}$ as a homoclinic point. 
	
	Consider this second case. Let $[V,V^{-n_0}]^s$ be the zigzag part of $W_X^s$ as in Lemma \ref{lem:zigzag}, let the point $M$ be as in Lemma \ref{lem:main_thm_lemma2}, $\delta_{i_0} \subset W_X^u$ as in Lemma \ref{lemma:unstable_yaxis_first}, and $\Delta \subset W_X^u$ as in (\ref{eq:Gamma_Delta}). We know that every post-critical point on $W_X^u$ is a forward iterate of a transverse intersection point of $W_X^u$ with the $y$-axis. These intersection points are all contained in polygons such that the boundary of each polygon consists of a polygonal line contained in $W_X^u$ and a polygonal line contained $[X,V^{-n_0-1}]^{s}$.  
	
	 More precisely, in the case when $\Delta$ does not intersect the $y$-axis, from Lemma \ref{lem:main_thm_lemma1} we obtain that the intersection points of $W_X^u$ with the $y$-axis are contained in the triangle $XVZ^1$. If $\Delta$ intersects the $y$-axis, by Lemma \ref{lem:main_thm_lemma2}(1) we see that these points are contained in the union $\mathcal{G}\cup L_{a,b}(\mathcal{G})$, where the boundary of the polygon $\mathcal{G}$ is $\partial\mathcal{G} = [X,M]^u\cup[M,X]^{s}$. Also notice that in both cases, the interiors of these polygons do not intersect $W_X^s$, and the only points in $W_X^u\cap\{y\text{-axis}\}$ that may lie in the boundary of these polygons are $Z^{-1}$ and, if $\Delta\cap\{y\text{-axis}\}\neq\emptyset$, the intersection of $\delta_{i_0}$ with the $y$-axis.
	 
	 Now, let $T$ be a post-critical point on $W_X^u$ that lies on $\overline{VV^1}^s$. Let $n\in\mathbb{N}$ be such that $W_X^u$ transversely intersects the $y$-axis at $T^{-n}$. Then by Lemma \ref{lemma:hom_tang}, $T^{-n}$ is also a V-point at which $L_{a,b}^{-n}\bigl(\overline{VV^1}^{s}\bigr)$ tangentially intersects $W_X^u$.
	 
	 However, we know that $T^{-n}$ is contained in a polygon as described above. Denote that polygon by $\mathcal{P}$. We distinguish between three cases: $T^{-n}\in\Int\mathcal{P}$, $T^{-n}\in\partial\mathcal{P}\cap W_X^u$, and $T^{-n}\in\partial\mathcal{P}\cap W_X^s$.
	 
	 If $T^{-n}\in\Int\mathcal{P}$, then $T^{-n}$ does not lie on $W_X^s$ since $\Int\mathcal{P}\cap W_X^s=\emptyset$.
	 
	 Now let $T^{-n}\in\partial\mathcal{P}\cap W_X^s$. Every V-point on $\partial\mathcal{P}\cap W_X^s$ is either (\emph{i}) a backward iterate of $V$ on the zigzag part of $W_X^s$ or (\emph{ii}) the V-point on $[V^{-n_0},V^{-n_0-1}]^s$ lying on the $y$-axis. If (\emph{i}), then we see that $V$ is a homoclinic point, and the claim of the theorem follows. In the case (\emph{ii}), Lemma \ref{lem:main_thm_lemma2}(2) implies that $T^{-n}$ coincides with $Z^{2i_0}$ or $Z^{2i_0+2}$, so we see that $Z$ is a homoclinic point.
	 
	 Finally, let $T^{-n}\in\partial\mathcal{P}\cap W_X^u$. Then $T^{-n}$ can be $Z^{-1}$, in which case $Z$ is a homoclinic point and the claim follows. Alternatively, $T^{-n}$ can be the intersection point of $\delta_{i_0}$ with the $y$-axis, in which case it again follows from Lemma \ref{lem:main_thm_lemma2}(2) that $T^{-n}$ is either $Z^{2i_0}$ or $Z^{2i_0+2}$. This finishes the proof.
	\end{proof}

	\subsection{Homoclinic intersections along a segment} \label{subsec:hom_intersec_segment}
	
	In order to prove the main Theorem \ref{thm:hom_T0V0}, it remains to consider the case when homoclinic points are not isolated in the relative topologies on $W_X^s$ and $W_X^u$. In other words, we consider parameter pairs for which $W_X^s$ and $W_X^u$ intersect along a (straight or polygonal) segment.
	
	We already know by Lemma \ref{lem:Ws+X_WuX_intersect_at_X_only} that all homoclinic intersections are contained in $W_X^{s-}$. The following lemma shows that these intersections are a proper subset of $W_X^{s-}$ and $W_X^u$.
	
	\begin{lem}  \label{lem:W_X_u_W_X_s-_do_not_coincide}
	$W_X^u$ and $W_X^{s-}$ do not coincide.
	\end{lem}
	
	\begin{proof}
	Seeking a contradiction, assume that the converse holds. Now observe the unstable manifold $W_Y^u$ of the other fixed point $Y$ in the third quadrant $\mathcal{Q}_3$, more precisely, the part of $W_Y^u$ which passes through $Y$ and lies on a straight line whose direction is the unstable eigenvector of $DL_{a,b}$ at $Y$. That part of $W_Y^u$ intersects the segment $\overline{VV^1}^s$ at the point
	\begin{equation} \label{eq:point_D}
	D=\left(\frac{a^2+2b-2-a\sqrt{a^2+4b}}{2\bigl(a^2-(1-b)^2\bigr)},\: \frac{-a^2+2b^2-2b-a\sqrt{a^2+4b}}{2\bigl(a^2-(1-b)^2\bigr)} \right).
	\end{equation}
	Therefore, $W_X^u$ and $W_X^{s-}$ coinciding would imply that the unstable manifolds of $X$ and $Y$ intersect, which is a contradiction. This finishes the proof.	 
	\end{proof}
	
	As a consequence, the set of homoclinic points for $X$ cannot contain any segments of the form $[A,A^1]^s$ for any $A\in W_x^{s-}$, nor those of the form $[B,B^2]^u$ for any $B\in W_X^u$. In particular, the set of homoclinic points for $X$ can only contain "true" segments (that have both endpoints).
	
	Recall the zigzag structure of $W_X^s$ from Lemma \ref{lem:zigzag}: $n_0$ denotes the smallest positive integer such that $V^{-n_0}$ lies in $\mathcal{Q}_2$, and $[V,V^{-n_0+1}]^s$ is a polygonal line contained in $\mathcal{Q}_3$.
	
	Denote by $E_0$ the intersection of $\overline{V^{-n_0+1}V^{-n_0}}^s$ with the negative $x$-axis. Then $[V^{-n_0},V^{-n_0-1}]^s$ is a polygonal line which intersects both the $y$-axis (at $E_0^{-1}$) and the negative $x$-axis at a point which we denote by $E_1$. In general, all subsequent preimages of $\overline{V^{-n_0+1}V^{-n_0}}^s$ intersect both the $y$-axis and the negative $x$-axis. For a positive integer $i$, we denote by $E_i$ the intersection point of $[V^{-n_0-i+1},V^{-n_0-i}]^s$ with the negative $x$-axis; see Figure \ref{figure:points_Eij}. Notice that for every $i\in\mathbb{N}_0$, $E_i^{-1}$ is a $V$-point on $W_X^{s}$ lying on the negative $y$-axis, $E_i^{-2}$ lies in $\mathcal{Q}_2$, and $E_i^{-3}$ lies in $\mathcal{Q}_4$ or $\mathcal{Q}_1$. Also notice that every $V$-point of $W_X^s$ in $\mathcal{Q}_4$ is either a backward iterate of $V$ or that of one of the points $E_i$. In addition, for every $i,j\in\mathbb{N}_0$, there exist unique $p,q\in\mathbb{N}$ such that $E_i^{-j}$ is contained in $[D^{-p},V^{-q}]^{s}\cup[V^{-q},D^{-p-1}]^{s}$, where $D$ is the intersection point of $W_X^{s}$ and $W_Y^{u}$ in $\mathcal{Q}_3$ defined by (\ref{eq:point_D}).
	
	\begin{figure}[!ht]
	\begin{center}
	\begin{tikzpicture}[auto, scale=0.6]
			\tikzstyle{nodec}=[draw,circle,fill=black,minimum size=2pt,
			inner sep=0pt, label distance=1.5mm]			

			\draw[->] (-15,0)--(12,0) node [below]{$x$};
			\draw[->] (0,-14.5)--(0,17) node [left]{$y$};
			\node[label={[xshift=-0.2cm, yshift=-0.6cm]$0$}] at (0,0) {};
			
			\coordinate (origin) at (0,0);
			\coordinate (ex) at (1,0);
			\coordinate (ey) at (0,1);
			
			
			\coordinate (start) at (3,2);
			\coordinate (v) at (0,-2.7);
			\coordinate (vm1) at (-2.29,-0.79);
			\coordinate (vm2) at (-1,-4);
			\coordinate (vm3) at (-3.86,-0.36);
			\coordinate (vm4) at (-0.28,-6.68);
			\coordinate (vm5) at (-7.26,3.59);
			\coordinate (vm6) at (3.85,-6.89);
			\coordinate (vm7) at (-9.21,13.11);
			\coordinate (vm8) at (8.59,-1.75);
			\coordinate (vm9) at (-2,14);
			\coordinate (vm10) at (6.88,4.06);
			
			\coordinate (e0) at (intersection of vm4--vm5 and origin--ex);
			\coordinate (e0m1) at (0,-8.74);
			\coordinate (e0m2) at (-10.93,8.05);
			\coordinate (e0m3) at (6.92,-6.09);
			\coordinate (e0m4) at (-8.38,15.89);
			\coordinate (e0m5) at (10.07,-1.11);
			
			\coordinate (e1) at (intersection of vm5--e0m1 and origin--ex);
			\coordinate (e1m1) at (0,-9.66);
			\coordinate (e1m2) at ($1.1018*(e0m2)-0.1018*(vm1)$);
			\coordinate (e1m3) at (7.6,-6.53);
			\coordinate (e1m4) at (-8.58,16.89);
			\coordinate (e1m5) at (10.96,-0.82);
			
			\coordinate (e2) at (intersection of e0m2--e1m1 and origin--ex);
			\coordinate (e2m1) at (0,-10.79);
			\coordinate (e2m2) at ($1.2093*(e0m2)-0.2093*(vm1)$);
			\coordinate (e2m3) at ($1.35*(e1m3)-0.35*(e0m3)$);
			
			\coordinate (e3) at (intersection of e1m2--e2m1 and origin--ex);
			\coordinate (e3m1) at (0,-11.71);
			\coordinate (e3m2) at ($1.29638*(e0m2)-0.29638*(vm1)$);
			\coordinate (e3m3) at ($2.1818*(e1m3)-1.1818*(e0m3)$);
			
			\coordinate (e4) at (intersection of e2m2--e3m1 and origin--ex);
			\coordinate (e4m1) at (0,-12.99);
			\coordinate (e4m2) at ($1.4468*(e0m2)-0.4468*(vm1)$);
			
			\coordinate (e5) at (intersection of e3m2--e4m1 and origin--ex);
			\coordinate (e5m1) at (0,-14);
			
			\coordinate (e6) at (intersection of e4m2--e5m1 and origin--ex);
			
			\draw[red, thick] (start) node[right]{$W_X^s$}--(v)--(vm1)--(vm2)--(vm3)--(vm4)--(vm5)--(e0m1)--(vm6)--(e1m1)--(e0m2)--(vm7)--(e1m2)--(e2m1)--(e0m3)--(vm8)--(e1m3)--(e3m1)--(e2m2)--(e0m4)--(vm9)--(e1m4)--(e3m2)--(e4m1)--(e2m3)--(e0m5)--(vm10)--(e1m5)--(e3m3)--(e5m1)--(e4m2);
			
			\node[nodec, color=red, label={[right, color=red]$V$}] at (v) {};
			\node[nodec, color=red, label={[above, color=red]$V^{-1}$}] at (vm1) {};
			\node[nodec, color=red, label={[right, color=red]$V^{-2}$}] at (vm2) {};
			\node[nodec, color=red, label={[above, xshift=1.5mm, color=red]$V^{-3}$}] at (vm3) {};
			\node[nodec, color=red, label={[right, color=red]$V^{-4}$}] at (vm4) {};
			\node[nodec, color=red, label={[above, color=red]$V^{-5}$}] at (vm5) {};
			\node[nodec, color=red, label={[above, color=red]$V^{-6}$}] at (vm6) {};
			\node[nodec, color=red, label={[right, color=red]$V^{-7}$}] at (vm7) {};
			\node[nodec, color=red, label={[left, color=red]$V^{-8}$}] at (vm8) {};
			\node[nodec, color=red, label={[below, color=red]$V^{-9}$}] at (vm9) {};
			\node[nodec, color=red, label={[above, color=red]$V^{-10}$}] at (vm10) {};
			
			\node[nodec, color=black, label={[above, color=black]$E_0$}] at (e0) {};
			\node[nodec, color=black, label={[below, color=black]$E_1$}] at (e1) {};
			\node[nodec, color=black, label={[above, color=black]$E_2$}] at (e2) {};
			\node[nodec, color=black, label={[below, color=black]$E_3$}] at (e3) {};
			\node[nodec, color=black, label={[above, color=black]$E_4$}] at (e4) {};
			\node[nodec, color=black, label={[below, color=black]$E_5$}] at (e5) {};
			\node[nodec, color=black, label={[above, color=black]$E_6$}] at (e6) {};
			
			\node[nodec, color=black, label={[right, color=black]$E_0^{-1}$}] at (e0m1) {};
			\node[nodec, color=black, label={[right, color=black]$E_1^{-1}$}] at (e1m1) {};
			\node[nodec, color=black, label={[right, color=black]$E_2^{-1}$}] at (e2m1) {};
			\node[nodec, color=black, label={[right, color=black]$E_3^{-1}$}] at (e3m1) {};
			\node[nodec, color=black, label={[right, color=black]$E_4^{-1}$}] at (e4m1) {};
			\node[nodec, color=black, label={[right, color=black]$E_5^{-1}$}] at (e5m1) {};
			
			\node[nodec, color=black, label={[right, color=black]$E_0^{-2}$}] at (e0m2) {};
			\node[nodec, color=black, label={[right, color=black]$E_1^{-2}$}] at (e1m2) {};
			\node[nodec, color=black, label={[right, color=black]$E_2^{-2}$}] at (e2m2) {};
			\node[nodec, color=black, label={[right, color=black]$E_3^{-2}$}] at (e3m2) {};
			\node[nodec, color=black, label={[right, color=black]$E_4^{-2}$}] at (e4m2) {};
			
			\node[nodec, color=black, label={[above, color=black]$E_0^{-3}$}] at (e0m3) {};
			\node[nodec, color=black, label={[below left, color=black]$E_1^{-3}$}] at (e1m3) {};
			\node[nodec, color=black, label={[above right, color=black]$E_2^{-3}$}] at (e2m3) {};
			\node[nodec, color=black, label={[right, color=black]$E_3^{-3}$}] at (e3m3) {};
			
			\node[nodec, color=black, label={[below, color=black]$E_0^{-4}$}] at (e0m4) {};
			\node[nodec, color=black, label={[left, color=black]$E_1^{-4}$}] at (e1m4) {};
			
			\node[nodec, color=black, label={[left, color=black]$E_0^{-5}$}] at (e0m5) {};
			\node[nodec, color=black, label={[below right, color=black]$E_1^{-5}$}] at (e1m5) {};

			
			\coordinate (Ystart) at (3.29,0);
			\coordinate (Yend) at (-7.66,-5.46);
			
			\draw[teal, thick] (Ystart)--(Yend) node[left]{$W_Y^u$};
			
			\coordinate (d) at (intersection of Ystart--Yend and start--v);
			\coordinate (dm1) at (intersection of Ystart--Yend and v--vm1);
			\coordinate (dm2) at (intersection of Ystart--Yend and vm1--vm2);
			\coordinate (dm3) at (intersection of Ystart--Yend and vm2--vm3);
			\coordinate (dm4) at (intersection of Ystart--Yend and vm3--vm4);
			\coordinate (dm5) at (intersection of Ystart--Yend and vm4--vm5);
			\coordinate (dm6) at (intersection of Ystart--Yend and e1--e0m1);
			\coordinate (dm7) at (intersection of Ystart--Yend and e2--e1m1);
			\coordinate (dm8) at (intersection of Ystart--Yend and e3--e2m1);
			\coordinate (dm9) at (intersection of Ystart--Yend and e4--e3m1);
			\coordinate (dm10) at (intersection of Ystart--Yend and e5--e4m1);
			\coordinate (dm11) at (intersection of Ystart--Yend and e6--e5m1);
			
			\node[nodec, color=teal, label={[above, color=teal]$D$}] at (d) {};
			\node[nodec, color=teal, label={[above, color=teal]$D^{-1}$}] at (dm1) {};
			\node[nodec, color=teal, label={[below right, color=teal]$D^{-2}$}] at (dm2) {};
			\node[nodec, color=teal, label={[above, color=teal]$D^{-3}$}] at (dm3) {};
			
			\node[nodec, color=teal, label={[above, color=teal]$D^{-9}$}] at (dm9) {};
			\node[nodec, color=teal, label={[below right, xshift=-1.5mm, color=teal]$D^{-10}$}] at (dm10) {};
			\node[nodec, color=teal, label={[left, color=teal]$D^{-11}$}] at (dm11) {};			

			\end{tikzpicture}
	\end{center}
	\caption{Sketch of points $E_i^{-j}$ (black) for parameters $a=1$, $b=0.95$. Point $D$ is the intersection of $W_X^s$ and $W_Y^u$ in the third quadrant $\mathcal{Q}_3$, defined as in (\ref{eq:point_D}).}
	\label{figure:points_Eij}
	\end{figure}
	
	Recall that $\partial\mathcal{H}$ denotes the boundary of the set of all parameters $(a,b)$, $a>0$, $0<b<1$, such that the Lozi map $L_{a,b}$ has homoclinic points for the fixed point $X$. 
	
	\begin{theorem} \label{thm:hom_segments}
	Let $(a,b)\in\partial\mathcal{H}$ and assume that $W_X^{s-}$ and $W_X^u$ intersect along a segment $\beta$. Then:
	\begin{enumerate}
		\item one endpoint of $\beta$ is an iterate of $Z$, and the other one is an iterate of $V$,
		\item the set of homoclinic points for $X$ is the union of orbits of points lying on $\beta$, that is, there are no other isolated tangential homoclinic points for $X$.
	\end{enumerate}
	\end{theorem}
	
	\begin{proof}
	We will first prove claim $(1)$. Let $W_X^u$ and $W_X^{s-}$ intersect along a segment $\beta$. By Lemma \ref{lem:W_X_u_W_X_s-_do_not_coincide}, we know that $\beta$ has two endpoints. In addition, since $\beta\subset W_X^{u}$, there exists a $k\in\mathbb{N}_0$ such that $L_{a,b}^{-k}(\beta)\subset\overline{XZ}^u$. Denote $\alpha:=L_{a,b}^{-k}(\beta)$. Then $\alpha$ is a straight line segment whose endpoints are $V$-points on $W_X^{s-}$. Therefore, the endpoints of $\alpha$ can be one of the following:
	\begin{enumerate}[(a)]
		\item backward iterates of $V$;
		\item backward iterates of intersections of $W_X^{s-}$ with the positive $x$-axis; since the intersection point of $W_X^{s-}$ with the positive $y$-axis also lies on $W_X^u$ as a homoclinic point, that point is $Z$;
		\item backward iterates of points $E_i$ (intersections of $W_X^{s-}$ with the negative $x$-axis, as defined above).
	\end{enumerate}
	
	We will now show that case (c) is not possible. Assume by contradiction that one of the endpoints of $\alpha$ is the point $E_i^{-j}$, for some $i,j\in\mathbb{N}_0$, and that this point does not coincide with any iterate of $Z$. Then there exists an $m\in\mathbb{N}$ such that $\alpha$ is contained in $\overline{Z^{-2m-2}Z^{-2m}}^u$, and $\alpha$ does not contain the endpoints of that segment. Consequently, $L_{a,b}^{2m}(\alpha)$ is a straight line segment lying on $\overline{ZZ^2}^u$ in $\mathcal{Q}_4$, and one of the endpoints of $L_{a,b}^{2m}(\alpha)$ is $E_i^{-j+m}$.
	
	Let the point $D$ be defined by (\ref{eq:point_D}), and let $p,q\in\mathbb{N}$ be such that $E_i^{-j+m}$ is contained in $[D^{-p},V^{-q}]^{s}\cup[V^{-q},D^{-p-1}]^{s}$. Observe the polygon $\mathcal{P}$ whose boundary is $\partial\mathcal{P}=[V^{-q},D^{-p-1}]^{s}\cup\overline{D^{-p-1}D^{-p}}\cup[D^{-p},V^{-q}]^{s}$. The point $E_i^{-j+m+2}$ lies outside of $\mathcal{P}$. Therefore, in order to pass through that point, $[Z^2,Z^4]^{u}$ will intersect the interior $\Int\mathcal{P}$ and consequently, it intersects transversely the boundary of $\mathcal{P}$. Since the boundary $\partial\mathcal{P}$ is contained in $W_X^s\cup W_Y^u$, $[Z^2,Z^4]^u$ can intersect only the portion of $\partial\mathcal{P}$ that belongs to $W_X^s$. This implies that $W_X^u$ and $W_X^s$ intersect transversely, which is a contradiction with Lemma \ref{lemma:hom_tang}. Therefore, backward iterates of points $E_i$ cannot be the endpoints of $\alpha$.
	
	It follows that the endpoints of $\alpha$ can only be iterates of $Z$ and $V$ (cases (a) and (b)). Due to Lemma \ref{lem:W_X_u_W_X_s-_do_not_coincide}, the endpoints of $\alpha$ cannot both be iterates of $Z$, nor can they both be iterates of $V$. Therefore, one endpoint is an iterate of $Z$, and the other one an iterate of $V$, which proves claim (1).
	
	Let us now prove claim (2). For every $n\in\mathbb{N}_0$, define
	\begin{equation*}
	\delta_n=[Z^{2n},Z^{2n+2}]^u, \quad \displaystyle\Delta=\bigcup_{n=0}^{\infty}\delta_n,
	\end{equation*}
as in (\ref{eq:gamma_n_delta_n}) and (\ref{eq:Gamma_Delta}). Notice that if $W_X^s$ and $W_X^u$ intersect along a segment, then $\Delta$ intersects the $y$-axis: otherwise, $Z$ would lie as a homoclinic point on $\overline{VV^1}^s$ like in Lemma \ref{lem:main_thm_lemma1} (see Figure \ref{figure:thm_hom_case1}), and there would not exist a segment of homoclinic points joining $Z$ with an iterate of $V$.
	
	Let $i_0$ be the smallest index $i\in\mathbb{N}_0$ such that $\delta_i$ intersects the $y$-axis. By Lemma \ref{lemma:unstable_yaxis_first}, we know that $\delta_{i_0}$ is a straight line segment, $\delta_{i_0}=\overline{Z^{2i_0}Z^{2i_0+2}}^u$, and $Z^{2i_0}$, $Z^{2i_0+2}$ lie in $\mathcal{Q}_4$, $\mathcal{Q}_3$, respectively. 
	
	Now let $V^k$ be an iterate of $V$ such that $W_X^u$ and $W_X^s$ intersect along $[Z^{2i_0}, V^k]^u=[Z^{2i_0},V^k]^s$. We claim that $V^k$ cannot lie in $\mathcal{Q}_3$. Assume by contradiction that the converse holds. Then the segment $[Z^{2i_0},V^k]^u$ intersects the $y$-axis. Since $[Z^{2i_0},V^k]^u\subset\delta_{i_0}$, it follows that $[Z^{2i_0},V^k]^u$ is itself a straight line segment. On the other hand, since $[Z^{2i_0},V^k]^s\subset W_X^s$ intersects the $y$-axis, it is a polygonal line which breaks at the intersection point with the $y$-axis. This a contradiction, so it follows that $V^k$ does not lie in $\mathcal{Q}_3$, that is, it lies in $\mathcal{Q}_4$.
	
	Considering the segment $[Z^{2i_0+2},V^{k+2}]^u=[Z^{2i_0+2},V^{k+2}]^s$ and using the same argument, we obtain that $V^{k+2}$ lies in $\mathcal{Q}_3$. Therefore, $V^{k}=V^{-n_0-1}$ and $V^{k+2}=V^{-n_0+1}$, where $n_0\in\mathbb{N}$ is such that $[V,V^{-n_0}]^s$ is the zigzag part of $W_X^s$ from Lemma \ref{lem:zigzag}.
	
	Now observe the polygon $\mathcal{T}$ with boundary
	\begin{equation*}
	\partial\mathcal{T}=[Z^{2i_0},V^{-n_0+1}]^s\cup[V^{-n_0+1},Z^{2i_0}]^u,
	\end{equation*}	 
and the union
	\begin{equation*}
	\displaystyle\mathcal{U}=\bigcup_{n=0}^{\infty}L_{a,b}^n(\mathcal{T}).
	\end{equation*}
Notice that $\mathcal{U}$ contains all intersections of $W_X^u$ with the $y$-axis apart from $Z^{-1}$, and the only V-points of $W_X^s$ in $\mathcal{U}$ are iterates of $V$ ($V,V^{-1},\ldots,V^{-n_0+1},V^{-n_0}$) and $E_0^{-1}$.
	
	Seeking a contradiction, assume that $W_X^s$ and $W_X^u$ intersect along a segment, and that there exists an isolated tangential homoclinic point $A$ for $X$. By possibly replacing $A$ with one of its iterates, we may assume that $A$ lies on $\overline{VX}^s$, in which case $A$ is a post-critical point on $W_X^u$. Thus, there exists a $k\in\mathbb{N}$ such that $A^{-k}$ is a transverse intersection point of $W_X^u$ with the $y$-axis, and consequently, $A^{-k}$ is a V-point on $W_X^s$.
	
	Since $A^{-k}\in\mathcal{U}$, that point is either $E_0^{-1}$ or an iterate of $V$. If $A^{-k}$ would be an iterate of $V$, then $A^{-k}$ would not be an isolated homoclinic point since all iterates of $V$ are endpoints of segments of homoclinic points by claim (1) of the theorem. If $A^{-k}$ would be $E_0^{-1}$, then $W_X^s$ would intersect $\delta_{i_0}$ tangentially at $E_0^{-1}$ and, as in the proof of claim (1), $\delta_{i_0+1}$ would intersect $\overline{V^{-n_0}E_0^{-1}}^s$ transversely. This is a contradiction, so it follows that there are no isolated tangential homoclinic points. This completes the proof of claim (2) and the proof of the theorem.     
	\end{proof}
	
	Figure \ref{fig:homoclinic_segments_example} presents an example of homoclinic intersections of $W_X^u$ and $W_X^s$ along a segment, for numerically obtained values of parameters.
	
	\begin{rem} \label{rem:misiurewicz_no_segments}
	The result obtained in Theorem \ref{thm:hom_segments} is especially surprising when compared to the parameters considered in \cite{misiurewicz1980strange}, in particular, the Misiurewicz parameter set.
	
	Consider parameter pairs from the set $\mathcal{M}_{1,2}=\{(a,b)\in\mathbb{R}^2\colon 0<b<1,\; a>0,\; a>b+1\}$. Assume that for one such parameter pair, $W_X^{u}$ and $W_X^s$ intersect along a segment. That segment contains a point $A$ such that $L_{a,b}$ is dif\mbox{}ferentiable at all points on the orbit of $A$: namely, the set of points for which $L_{a,b}$ is not dif\mbox{}ferentiable at all points of their orbits is countable, and the aforementioned segment contains uncountably many points in the plane.
	
	The hyperbolicity theorem \cite[Theorem~3]{misiurewicz1980strange} now implies that the tangent space at the point $A$ can be split into one-dimensional stable and unstable subspaces, which contradicts the assumption that $W_X^s$ and $W_X^u$ coincide along a segment containing $A$. Therefore, for $(a,b)\in\mathcal{M}_{1,2}$, the stable and unstable manifold of $X$ do not intersect along a segment.   
	\end{rem}
	
	\begin{figure}
	\begin{center}
		\includegraphics[width=\linewidth]{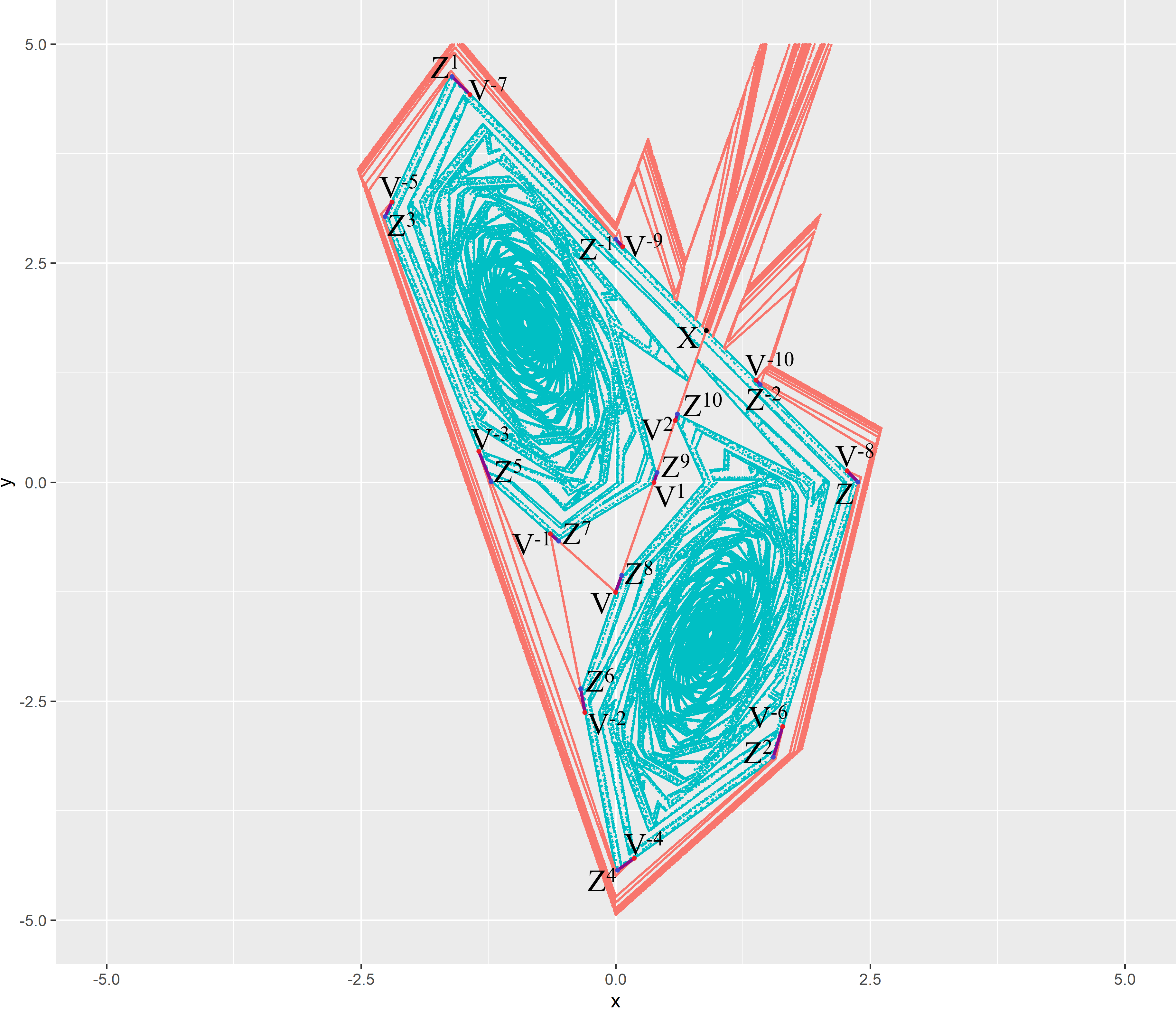}
	\end{center}
	\caption{The stable (red) and unstable (blue) manifold of $X$ for numerically obtained parameter values $a=1.0990404022941$, $b=0.975012556912299$. Every homoclinic point for $X$ belongs to the orbit of a point lying on the segment $\overline{VZ^8}$.}
	\label{fig:homoclinic_segments_example}
	\end{figure}
	
	We conclude this section with the proof of the main theorem.
	
	\begin{proof}[Proof of Theorem \ref{thm:hom_T0V0}]
	Let $(a,b)\in\partial\mathcal{H}$. For such a parameter pair, Lemma \ref{lemma:hom_tang} implies that $W_X^u$ and $W_X^s$, apart from $X$, intersect tangentially or along a segment.
	
	If all intersections of $W_X^s$ and $W_X^u$ are tangential, Theorem \ref{thm:hom_tangential} implies that the set of homoclinic points is the orbit of $Z$, orbit of $V$, or the union of these orbits. 
	
	If $W_X^s$ and $W_X^u$ intersect along a segment, Theorem \ref{thm:hom_segments} implies that the set of homoclinic points consists exactly of those points whose orbit intersects a segment joining an iterate of V with an iterate of $Z$. By possibly iterating this segment backward or forward a sufficient number of times, we conclude that the orbit of every homoclinic point also intersects a segment joining $V$ with an iterate of $Z$. This completes the proof.
	\end{proof}

	\section{Boundary curves in the parameter space}\label{sec:examples_bd_curves}
	
	After considering all possible homoclinic points for $X$ in the border case of their existence, we will now compute and describe some curves in the parameter space which represent a portion of that border, i.e., the locus of tangential homoclinic points for $X$ (Lemma \ref{lemma:hom_tang}). We obtain the equations of those curves in the form
	\begin{equation*}
	C_n\ldots\quad P_n(a,b)+Q_n(a,b)\sqrt{a^2+4b}=0,
	\end{equation*}
	where $P_n$ and $Q_n$ are polynomials in $a$ and $b$. The values of some $P_n$ and $Q_n$ are presented in Appendix \ref{appendix:curves_Cn_equations}. Since these equations can be written in the form $(P_n(a,b))^2-(Q_n(a,b))^2(a^2+4b)=0$, we see that these curves are algebraic curves. They are shown in the parameter space in Figure \ref{figure:homoclinic_border_curves}, and the corresponding configurations of $W_X^s$ and $W_X^u$ for a few boundary curves can be seen in Figure \ref{figure:homoclinic_border_WsWu_config}.
	
	\begin{figure}[!ht]
	\begin{center}
	\includegraphics[scale=1]{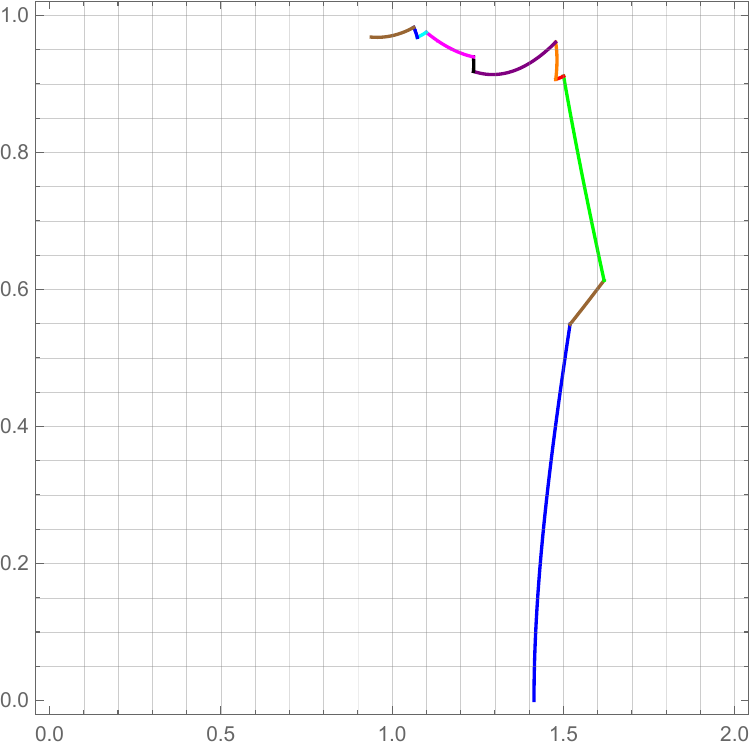}
	\end{center}
	\caption{Borders for the region of existence of homoclinic points for $X$: curves $C_1$ (blue), $C_2$ (brown), $C_3$ (green), $C_4$ (red), $C_5$ (orange), $C_6$ (purple), $C_7$ (black), $C_8$ (magenta), $C_9$ (cyan), $C_{10}$ (blue) and $C_{11}$ (brown). The values of parameter $a$ are presented on the horizontal and those of $b$ on the vertical axis.}	
	\label{figure:homoclinic_border_curves}
	\end{figure}
	
	By Theorem \ref{thm:hom_T0V0}, we know that the set of homoclinic points in the border case contains iterates of $Z$ or $V$. Moreover, Lemma \ref{lemma:hom1} implies that it suf{}fices to consider homoclinic points on the straight line segment $\overline{XV}^s$ in order to describe the boundary curves $C_n$.
	
	In the degenerate case $b = 0$, the Lozi map is a tent map, $L_{a,0} = T_a$, and a tangential homoclinic point appears for $a = \sqrt{2}$ when the second iterate of the critical value of the tent map coincides with the positive fixed point, i.e., in notation of Lozi maps, when $Z^2 = X$. Since for $b > 0$, the point $Z^2$ lies in the lower half-plane, the equation of the first boundary curve $C_1$ can be derived from the condition
	\begin{equation*}
	Z^2 \in [V, V^1]^{s}.
	\end{equation*}	  
	The corresponding part of the border of existence of homoclinic points for $X$ is only a portion of the curve $C_1$, starting at $(a_0,b_0)=(\sqrt{2},0)$, and ending at $(a_1,b_1)$ for which $Z^2$ coincides with $V$; see Figure \ref{figure:homoclinic_border_WsWu_config}(a).

	The equation of the second boundary curve $C_2$ is obtained from the condition
	\begin{equation*}
	V \in [Z^2, T]^{u},
	\end{equation*}	 
where $T$ is the first post-critical point on $[Z^2,Z^4]^u$, counting from $Z^2$. The endpoints of the curve $C_2$ are $(a_1,b_1)$ and $(a_2,b_2)$, a point in the parameter space for which $Z^4$ lies on $[V,V^{1}]^s$, and $V$ still lies on $[Z^2,T]^u$; see Figure \ref{figure:homoclinic_border_WsWu_config}(b) and (c).

	The third boundary curve $C_3$ is determined by
	\begin{equation*}
	Z^4 \in [V,V^{1}]^{s},
	\end{equation*}
and apart from $(a_2,b_2)$, the other endpoint of that curve is $(a_3,b_3)$ for which the homoclinic point $Z^4$ remains on $[V,V^{1}]^{s}$, but $Z^2$ also lies on the $y$-axis; see Figure \ref{figure:homoclinic_border_WsWu_config}(d) and (e).

	In general, on the boundary curves, for some $i \in \mathbb{N}$, at least one of the following conditions is satisfied:
	\begin{itemize}
		\item $Z^{2i} \in [V,V^1]^{s}$,
		\item $V \in [Z^{2i},T]^{u}$, where $T$ is the first post-critical point on $[Z^{2i},Z^{2i+2}]^{u}$ counting from $Z^{2i}$.
	\end{itemize}
	We remark that for some parameter pairs on the boundary curves, $W_X^u$ and $W_X^s$ may intersect along a segment, as it is shown in Figure \ref{fig:homoclinic_segments_example} (for parameter values on the curve $C_9$) and discussed in Subsection \ref{subsec:hom_intersec_segment}.
	
In addition, for a fixed $i$, the first condition can be satisfied on more than one boundary curve. Moreover, for every $n \in \mathbb{N}$, the endpoint $(a_n,b_n)$ of the boundary curves $C_{n}$ and $C_{n+1}$ is a parameter pair for which either both of the above conditions are satisfied (as in Figure \ref{figure:homoclinic_border_WsWu_config}(c)), $Z^{2j}$ lies on the $y$-axis (Figure \ref{figure:homoclinic_border_WsWu_config}(e)), or $V=Z^{2j}$ (Figure \ref{figure:homoclinic_border_WsWu_config}(f)), for some $j \in \mathbb{N}$. The approximative values of $a_n$ and $b_n$ for some boundary curves are given in Table \ref{table:hom_borders_anbn}.

\begin{table}
	\caption{Approximative values of $(a_n,b_n)$, the intersection points of $C_n$ and $C_{n+1}$ in the parameter space.}
		\renewcommand{\arraystretch}{1.25}
		\begin{tabular}{|c|c|c|}
		\hline
		$C_n$ & $a_n$ & $b_n$ \\
		\hline \hline
		$C_1$ & $1.51950144$ & $0.549133899$ \\
		\hline
		$C_2$ & $1.61870652$ & $0.613234325$ \\
		\hline
		$C_3$ & $1.50065366$ & $0.911203728$ \\
		\hline
		$C_4$ & $1.4778227$ & $0.906571953$ \\
		\hline
		$C_5$ & $1.47728163$ & $0.960699507$ \\
		\hline
		$C_6$ & $1.23772202$ & $0.918152634$ \\
		\hline
		$C_7$ & $1.23744761$ & $0.939287819$ \\
		\hline
		$C_8$ & $1.09974148$ & $0.975240539$ \\
		\hline
		$C_9$ & $1.07356315$ & $0.967833586$ \\
		\hline
		$C_{10}$ & $1.06317799$ & $0.982377305$ \\
		\hline
		$C_{11}$ & $0.939255378$ & $0.968217486$ \\
		\hline
		\end{tabular}
		\renewcommand{\arraystretch}{1}
	\label{table:hom_borders_anbn}
	\end{table}

	\begin{rem}
	If we write down the equation $P_1(a,b)+Q_1(a,b)\sqrt{a^2+4b}=0$ of the curve $C_1$ in its algebraic form, we obtain
	\begin{equation*}
	2a^4-4a^2-3a^2b+4b^3=0.
	\end{equation*}
This equation corresponds to the curve
	\begin{equation*}
	2a=\sqrt{3b^2+4+\sqrt{(3b^2+4)^2-32b^3}}
	\end{equation*}
used in \cite{misiurewicz1980strange} as the fourth condition for the construction of the Misiurewicz parameter set.
	\end{rem}        

	\begin{figure}[hbt!]
		\centering
		\begin{tabular}{cc}
			\raisebox{4cm}{(a)} \includegraphics[trim=7cm 7cm 7cm 7cm, clip=true,width=7cm,height=4.6cm]{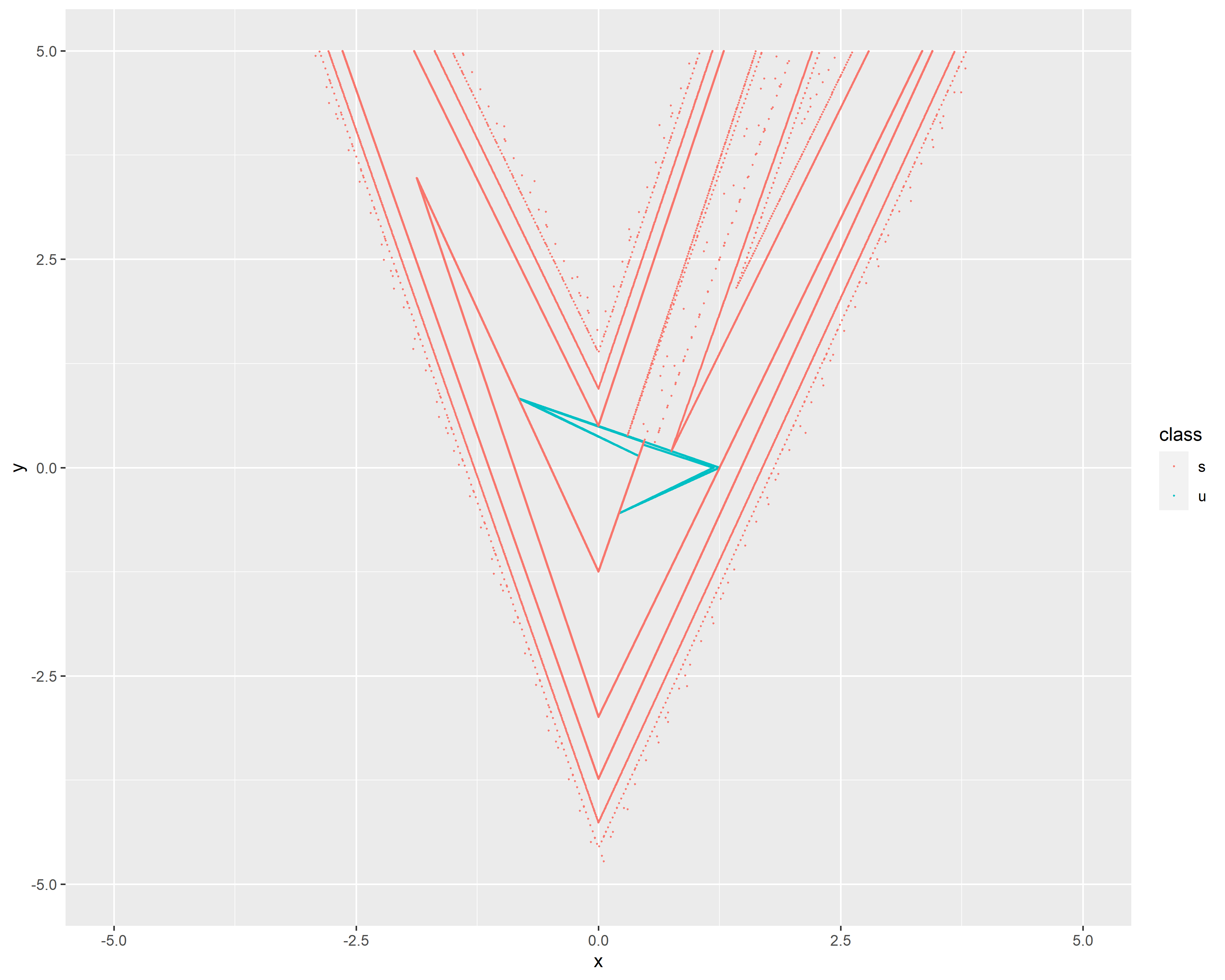} & 
			\raisebox{4cm}{(b)} \includegraphics[trim=6cm 6cm 6cm 6cm, clip=true,width=7cm,height=4.6cm]{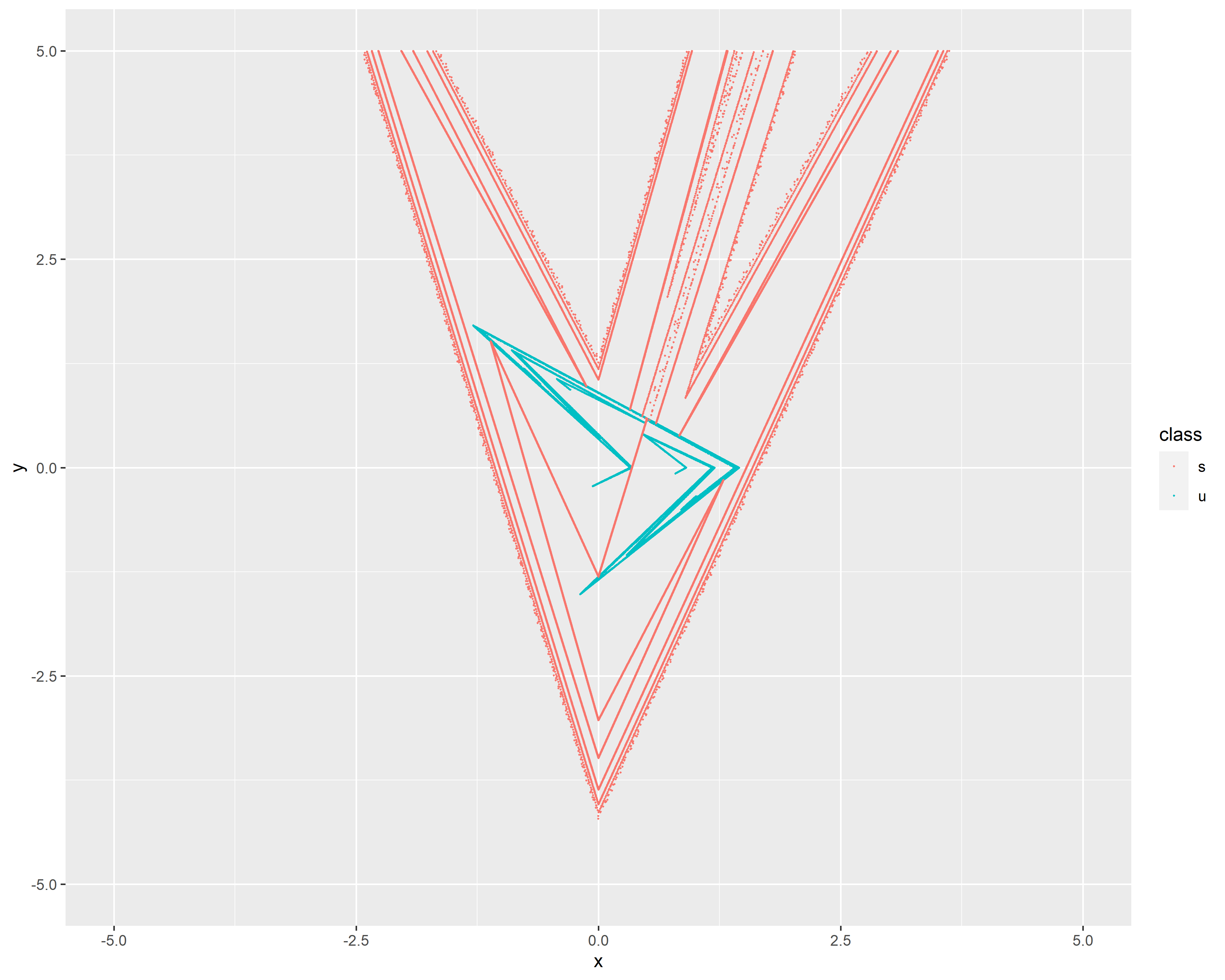} \\
			\raisebox{4cm}{(c)} \includegraphics[trim=6cm 4.5cm 6cm 6cm, clip=true,width=7cm,height=4.6cm]{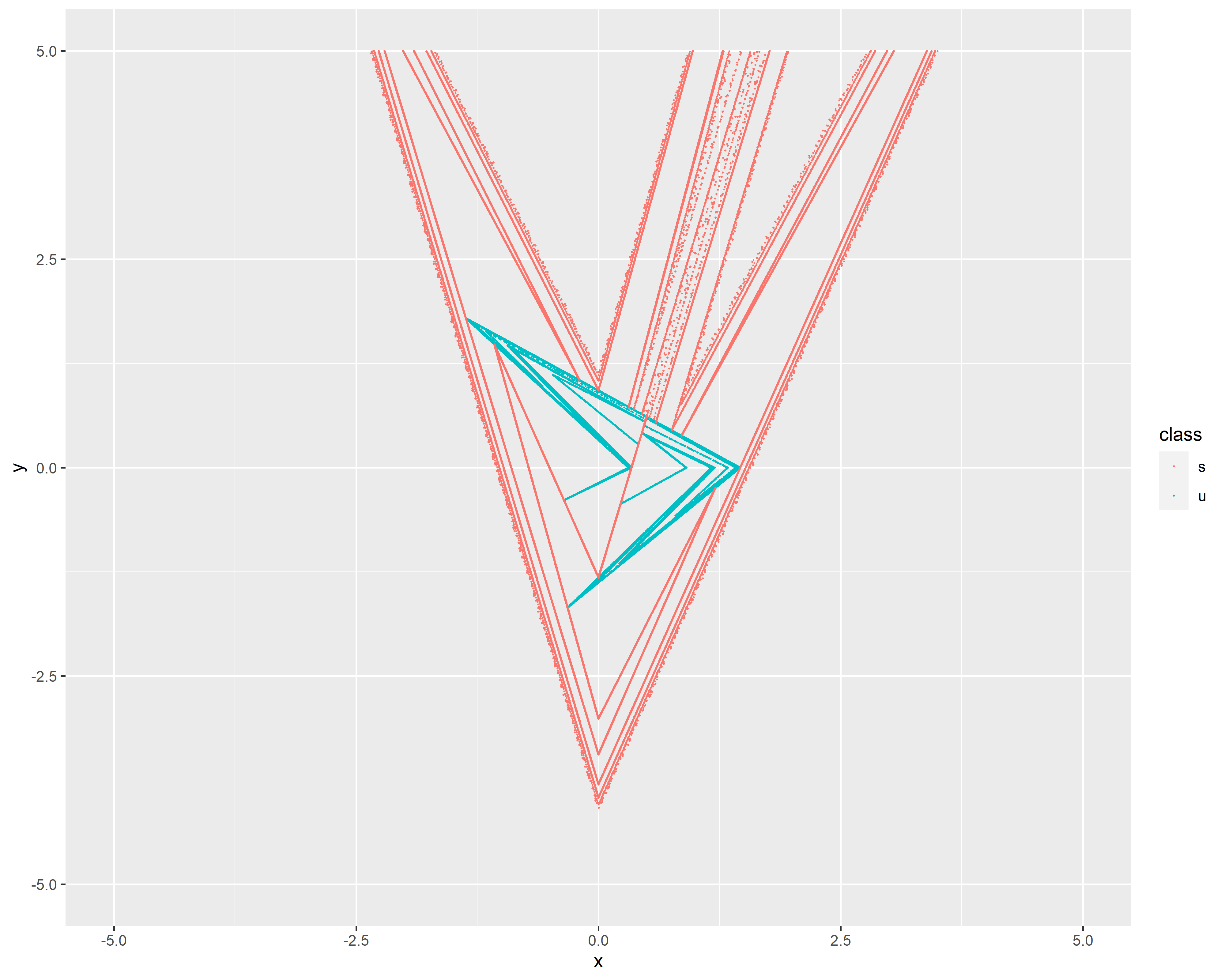} 
			& \raisebox{4cm}{(d)} \includegraphics[trim=6cm 6cm 6cm 5cm, clip=true,width=7cm,height=4.6cm]{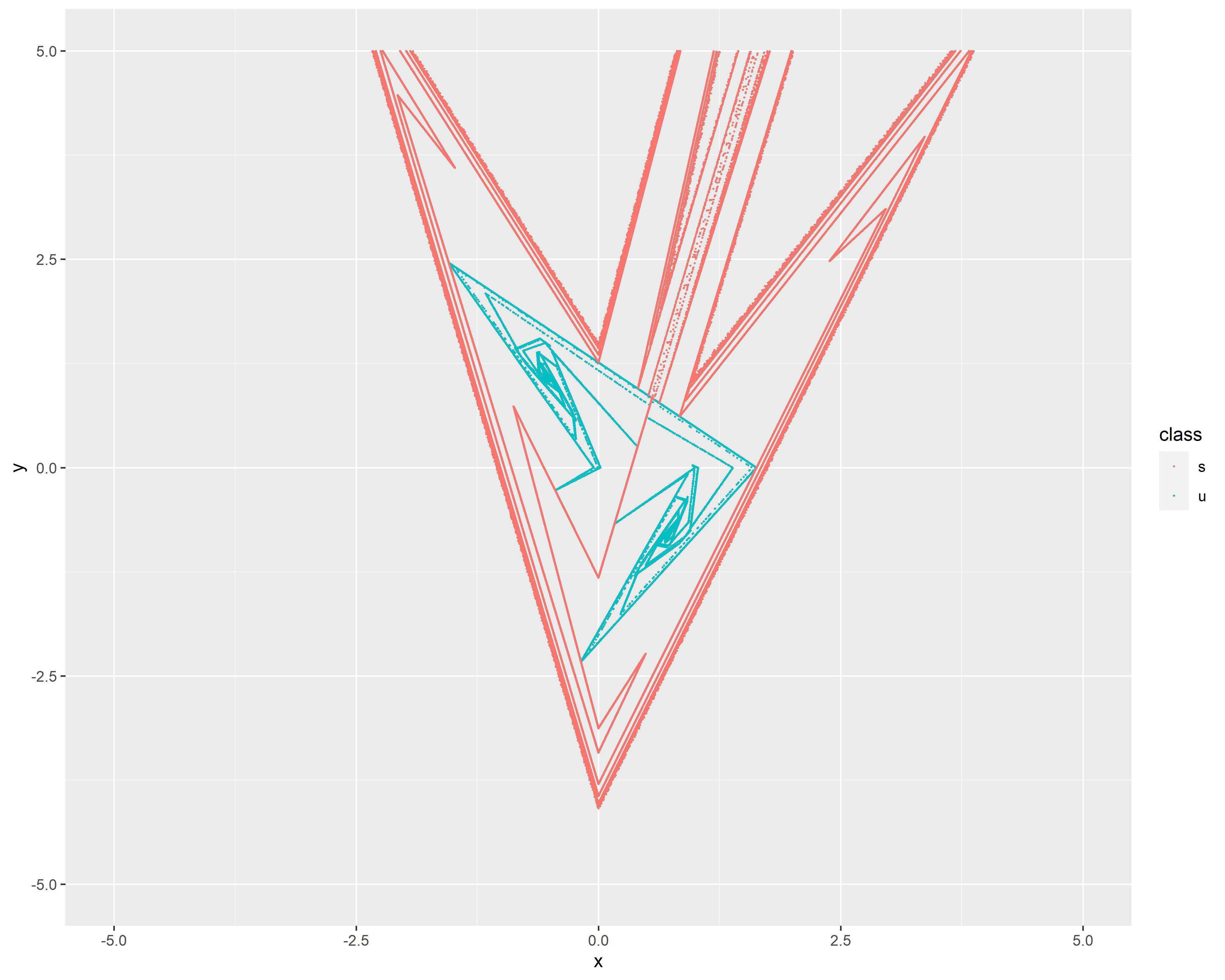} \\
			\raisebox{4cm}{(e)} \includegraphics[trim=6cm 3cm 6cm 3cm, clip=true,width=7cm,height=4.6cm]{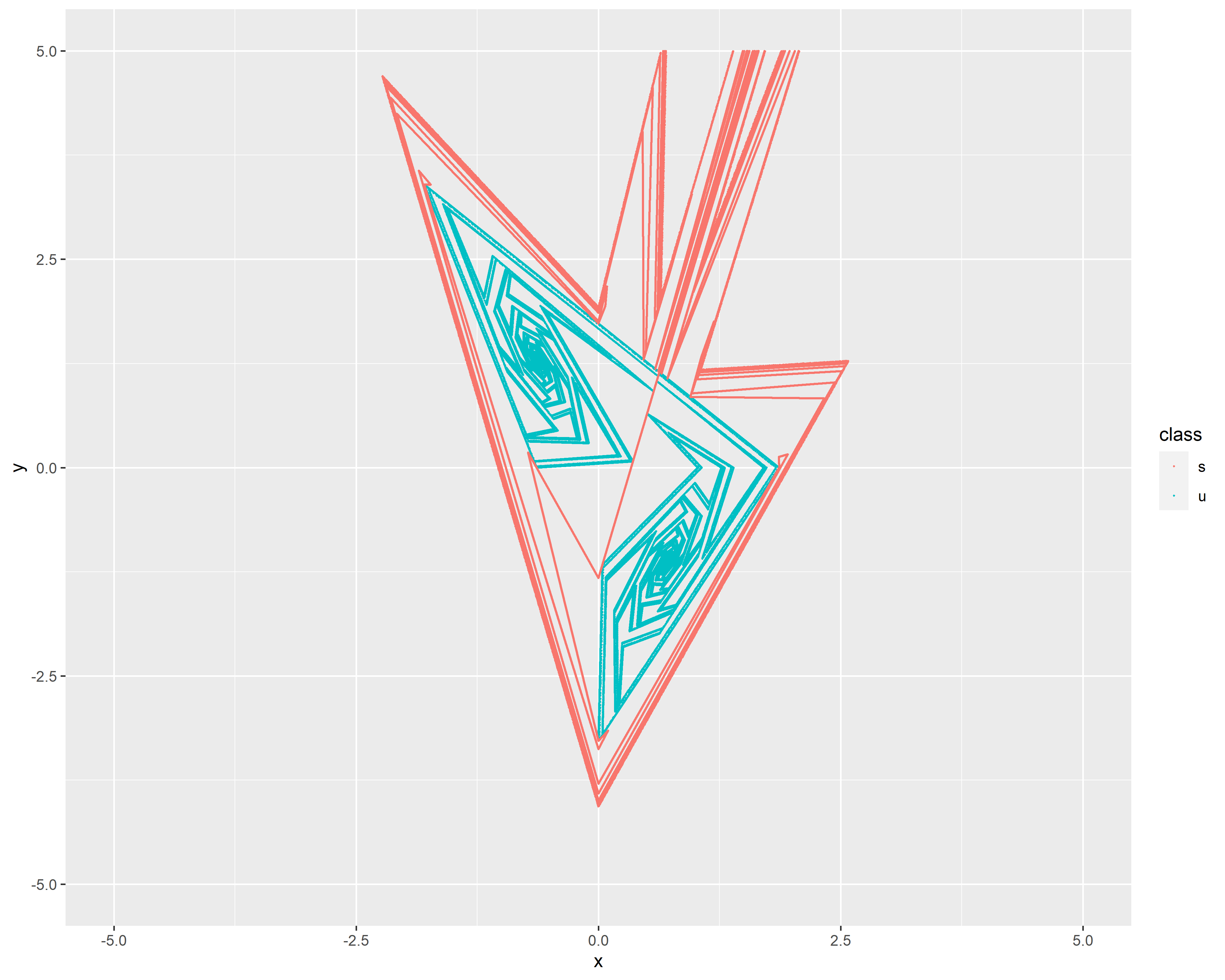} 
			& \raisebox{4cm}{(f)} \includegraphics[trim=4cm 3.5cm 4cm 2.5cm, clip=true,width=7cm,height=4.6cm]{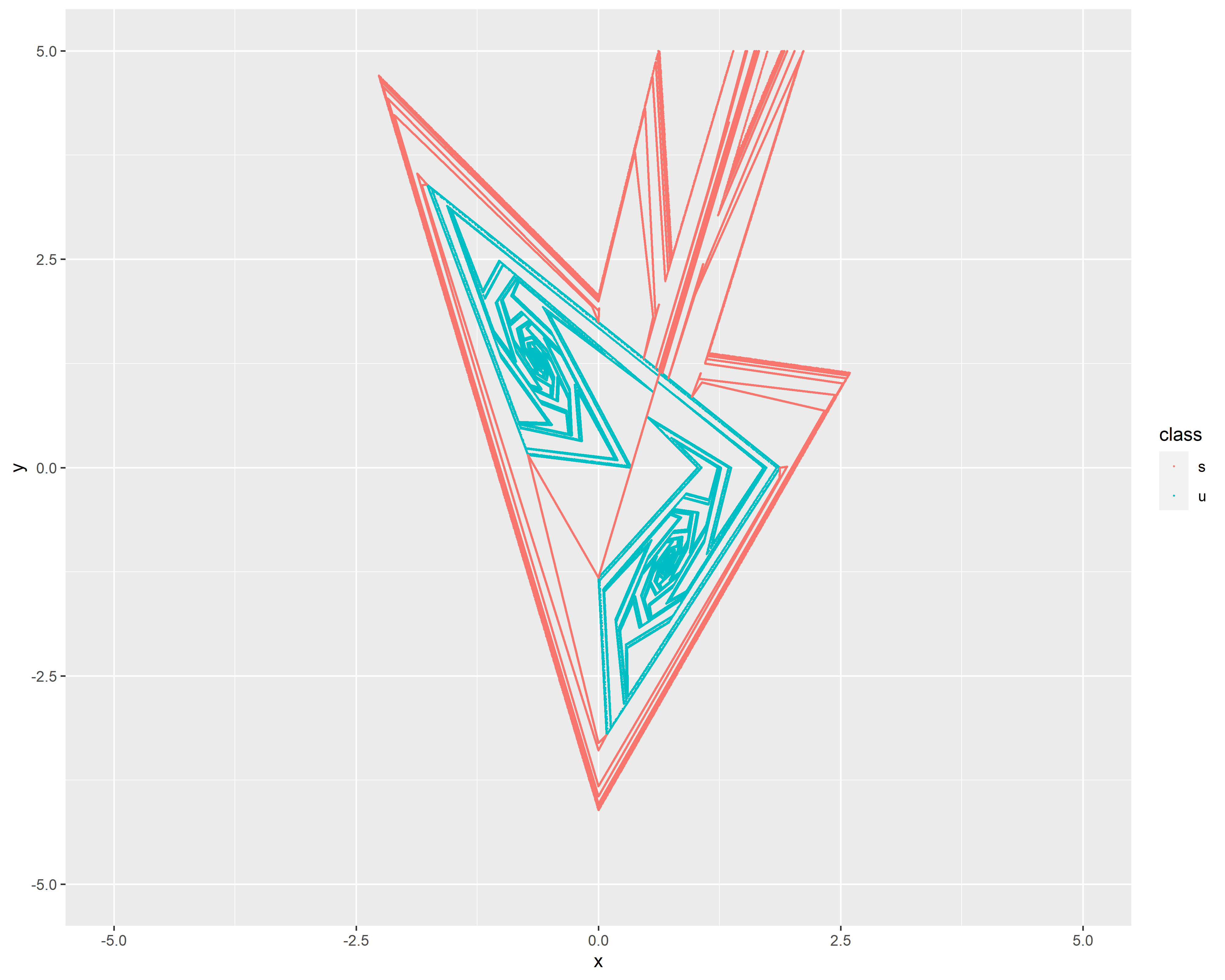} \\
			\raisebox{4cm}{(g)} \includegraphics[trim=6cm 4.5cm 6cm 3.5cm, clip=true,width=7cm,height=4.6cm]{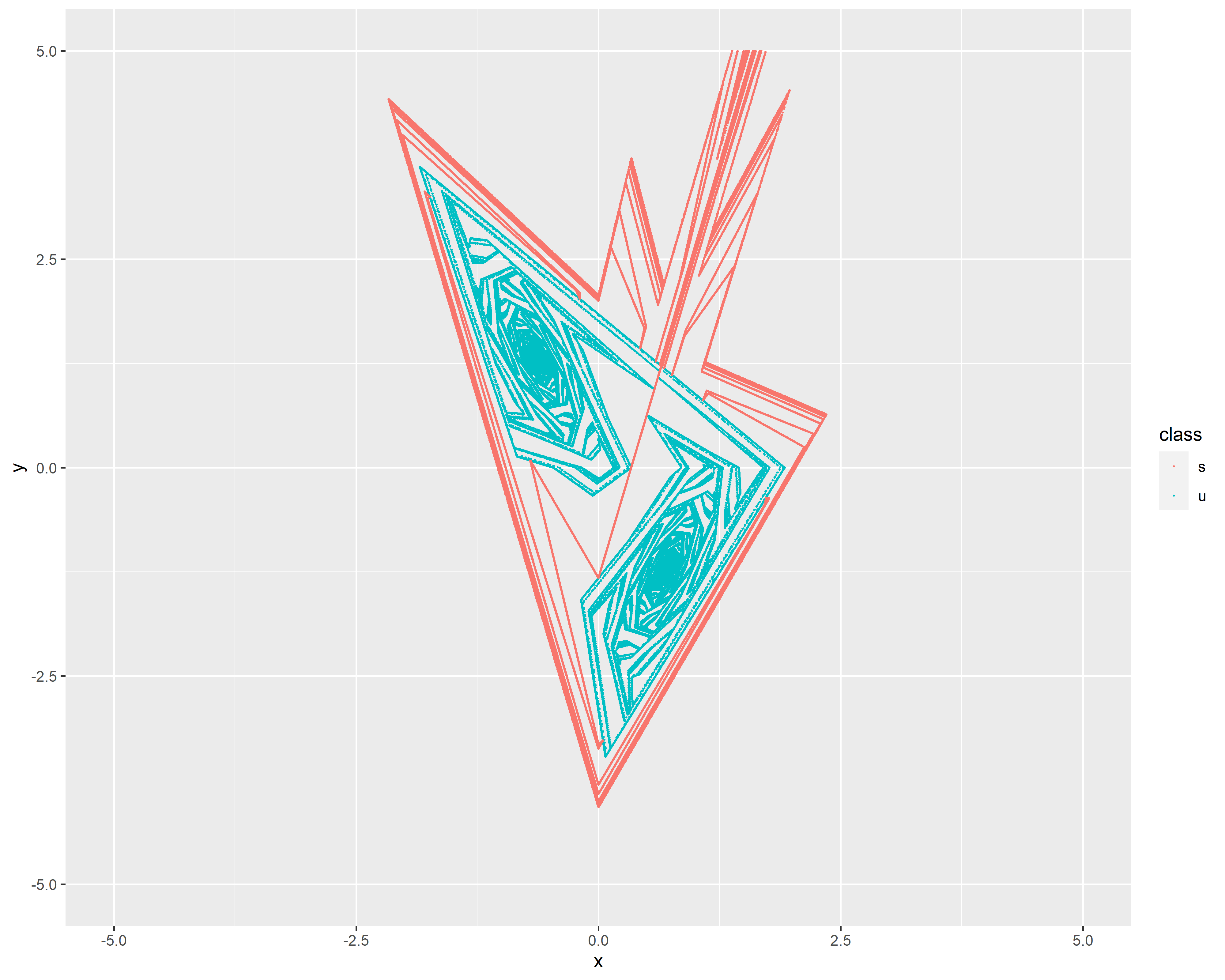}
			& \raisebox{4cm}{(h)} \includegraphics[trim=4cm 4cm 4cm 3cm, clip=true,width=7cm,height=4.6cm]{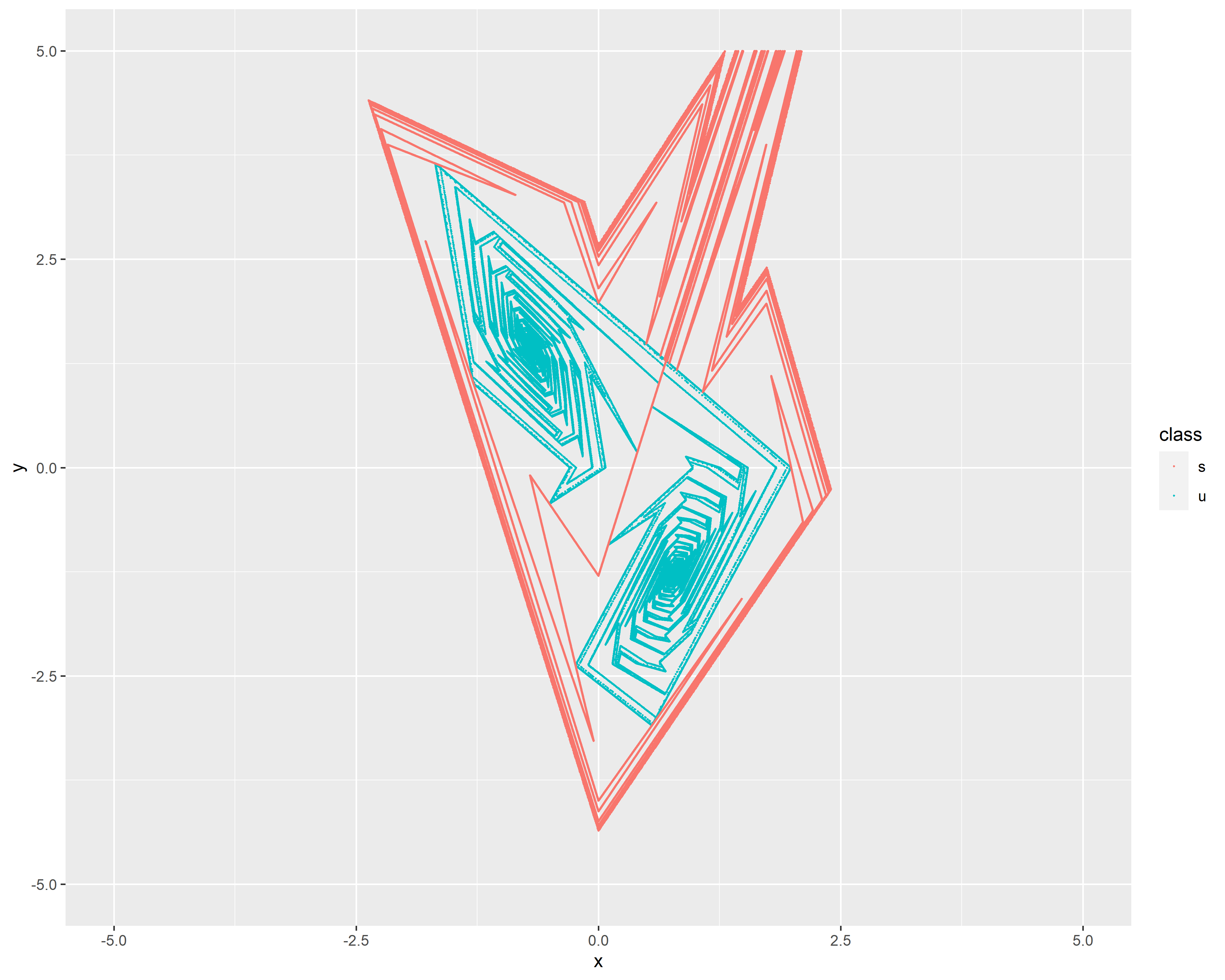}
		\end{tabular}
		\caption{$W_X^s$ (red) and $W_X^u$ (blue) for parameter pairs on the boundary curves: (a) curve $C_1$ ($a=1.46, b=0.332873$), (b) curve $C_2$ ($a=1.58, b=0.587775$), (c) endpoint $(a_2,b_2)$, (d) curve $C_3$ ($a=1.56, b=0.75378$), (e) endpoint $(a_3,b_3)$, (f) endpoint $(a_4,b_4)$, (g) curve $C_5$ ($a=1.48115, b=0.94$), (h) curve $C_6$ ($a=1.35, b=0.918178$).}
		\label{figure:homoclinic_border_WsWu_config} 
	\end{figure}

\newpage

	\appendix
	
	\section{Polynomials $P_n$ and $Q_n$ for $n=1,2,\ldots,11$}\label{appendix:curves_Cn_equations} 
	
	As already stated in Section \ref{sec:examples_bd_curves}, the equation of each boundary curve $C_n$ is obtained in the form $P_n(a,b)+Q_n(a,b)\sqrt{a^2+4b}=0$, where polynomials $P_n$ and $Q_n$ for $n=1,2,\ldots,11$ are as follows:
	
		\begin{align*}
		P_1 & (a,b) =a^3 - 4a, \\
		P_2 & (a,b) =4ab^5 - (8a^3 - 4)b^4 - (4a^5 + 8a^3 + 15a^2 + 4a + 4)b^3 \\
			  & \; + (15a^4 + 16a^3 + 11a^2)b^2 + (4a^7 + 2a^6 - 8a^5 - 10a^4)b - (2a^8 - 2a^6), \\
		P_3 & (a,b) =4b^3 + 3a^2b^2 - (a^4 + 6a^2 + 4a)b - (4a^4 + 4a^3 + 2a^2), \\
		P_4 & (a,b) =4ab^3 + (-9a^3 + 4a)b^2 + (2a^5 - 4a^3 + 4a)b + (4a^5 - 2a^3), \\
		P_5 & (a,b) =7a^2b^7 + (-11a^4 + 9a^2)b^6 + (3a^6 - 27a^4 + 11a^2 + 6a)b^5 + (20a^6 - 52a^4 - 35a^3)b^4 \\
			  & \; + (-4a^8 + 64a^6 + 56a^5)b^3 + (-28a^8 - 36a^7)b^2 + (4a^{10} + 10a^9)b - a^{11}, \\
		P_6 & (a,b) =4ab^3 + (a^3 - 4a)b^2 + (4a^3 - 4a)b + (-a^7 + 6a^5 - 6a^3 - 6a^2 - 4a), \\
		P_7 & (a,b) =4ab^{10} + (-2a^3 - 6a^2 + 4a)b^9 + (-16a^5 - 8a^4 - 12a^3 - 2a^2 + 4a)b^8 \\
			  & \; + (6a^7 + 10a^6 - 22a^5 + 2a^4 - 26a^3 - 16a^2 + 8a + 4)b^7 \\
			  & \; + (-2a^8 + 32a^7 + 10a^6 - 16a^5 + 46a^4 + 2a^3 - 11a^2 - 12a)b^6 \\
			  & \; + (-8a^9 - 16a^8 + 80a^7 + 72a^6 - 64a^5 - 10a^4 + 22a^3)b^5 \\
			  & \; + (4a^{10} - 48a^9 - 160a^8 - 16a^7 + 60a^6 + 68a^5)b^4 + (8a^{11} + 80a^{10} + 132a^9 - 6a^8 - 72a^7)b^3 \\
			  & \; + (-12a^{12} - 76a^{11} - 64a^{10} - 26a^9)b^2 + (12a^{13} + 38a^{12} + 32a^{11})b - (6a^{14}+6a^{13}), \\
		P_8 & (a,b) =(-4a + 4)b^6 + (-15a^3 + a^2 + 4a)b^5 + (-24a^5 - 13a^4 + 4a^3 - 2a^2 + 4a)b^4 \\
			  & \; + (-5a^7 - 19a^6 - 6a^5 + 2a^4 + 14a^3)b^3 + (6a^9 - 24a^7 + 4a^6 + 24a^5 + 8a^4 + 4a^3)b^2 \\
			  & \; + (2a^{11} + 2a^{10} - 8a^9 + 12a^8 + 8a^7 - 4a^6 + 4a^4 + 4a^3 + 2a^2)b \\
			  & \; + (4a^{10} - 4a^8 - 4a^7 - 4a^6 - 4a^5 - 2a^4), \\
		P_9 & (a,b) =-4b^6 + (17a^3 + 13a^2 - 4a)b^5 + (-4a^5 + 15a^4 + 8a^3 - 6a^2 - 4a)b^4 \\
			  & \; + (-13a^7 - 3a^6 + 26a^5 - 2a^4 - 2a^3 - 2a^2 - 4a)b^3 \\
			  & \; + (-2a^9 - 8a^8 + 8a^7 - 12a^6 + 8a^5 - 8a^3)b^2 \\
			  & \; + (2a^{11} + 2a^{10} - 8a^9 - 4a^8 + 8a^7 - 4a^6 - 8a^5 - 4a^4 - 4a^3 - 2a^2)b \\
			  & \; + (4a^{10} - 4a^8 + 4a^6 + 4a^5 + 2a^4), \\
		P_{10} & (a,b) =(-2a^2 + 4a)b^{12} + (-16a^4 - 14a^3 + 16a^2)b^{11} + (2a^6 - 30a^4 - 12a^3 + 16a^2 + 4a)b^{10} \\
				 & \; + (2a^8 + 2a^7 + 12a^6 + 6a^5 - 34a^4 + 6a^3 + 10a^2 + 4a)b^9 \\ 
				 & \; + (2a^8 + 4a^7 + 52a^6 + 12a^5 - 58a^4 - 14a^3 - 14a^2 - 12a + 4)b^8 \\
				 & \; + (-2a^{10} - 2a^9 - 10a^8 + 122a^6 + 168a^5 + 102a^4 + 88a^3 + 5a^2 + 16a - 4)b^7 \\
				 & \; + (-6a^{10} - 8a^9 - 74a^8 - 240a^7 - 226a^6 - 198a^5 - 84a^4 - 93a^3 + a^2 + 4a + 4)b^6 \\
				 & \; + (2a^{12} + 2a^{11} + 94a^9 + 260a^8 + 262a^7 + 116a^6 + 22a^5 + 77a^4 + 20a^3 - 21a^2 - 20a - 4)b^5 \\
				 & \; + (10a^{12} + 14a^{11} - 118a^{10} - 230a^9 - 160a^8 + 110a^7 - 24a^6 - 156a^5 - 49a^4 \\
				 & \quad + 81a^3 + 37a^2 + 8a)b^4 \\
				 & \; + (-2a^{14} - 14a^{13} + 92a^{11} + 137a^{10} - 58a^9 - 160a^8 + 84a^7 + 205a^6 - 8a^5 - 102a^4 - 40a^3)b^3 \\
				 & \; + (2a^{15} + 12a^{14} - 4a^{13} - 44a^{12} - 25a^{11} + 105a^{10} + 82a^9 - 70a^8 - 79a^7 - 5a^6 + 6a^5)b^2 \\
				 & \; + (-2a^{16} - 6a^{15} + a^{14} + 20a^{13} - 7a^{12} - 68a^{11} - 39a^{10} + 52a^9 + 59a^8 + 18a^7)b \\
				 & \; + (a^{17} + a^{16} - 3a^{15} - 3a^{14} + 12a^{13} + 13a^{12} - 9a^{11} - 15a^{10} - 5a^9), \\
		P_{11} & (a,b) =(8a - 4)b^6 + (-7a^3 + 5a^2 - 4a)b^5 + (-4a^5 + 34a^4 + 16a^3 + 2a^2 - 4a)b^4 \\
				 & \; + (-7a^7 - 9a^6 + 40a^5 - 32a^4 - 2a^3 + 6a^2 - 4a)b^3 \\
				 & \; + (-5a^8 - 4a^7 - 26a^6 + 16a^5 - 6a^4 - 16a^3)b^2 \\
				 & \; + (a^{11} + a^{10} - 6a^9 + 6a^8 + 6a^7 - 14a^6 - 6a^5 + 8a^4 + 4a^3)b \\
				 & \; + (4a^{10} - 4a^8 + 4a^6 + 4a^5 + 4a^4 + 4a^3 + 2a^2), \\
			   &   \\
		Q_1 & (a,b) =a^2 - 2b, \\
		Q_2 & (a,b) =(-4a^4 + 4a^2 - a)b^3 - (8a^4 + a^3 - 3a)b^2 + (4a^6 + 6a^5 - 6a^3)b - (2a^7 - 2a^5), \\
		Q_3 & (a,b) =-3ab^2 - (a^3 - 2a)b + (4a^3 + 4a^2 + 2a), \\
		Q_4 & (a,b) =2b^3 - 5a^2b^2 + (2a^4 + 4a^2)b + (-4a^4 + 2a^2), \\
		Q_5 & (a,b) =ab^7 + (-3a^3 + a)b^6 + (a^5 - a^3 + a + 1)b^5 - 9a^2b^4 + 24a^4b^3 - 22a^6b^2 + 8a^8b - a^{10}, \\
		Q_6 & (a,b) =-2b^3 + 3a^2b^2 + 2a^4b + (-a^6 - 2a^4 + 2a^2 + 2a), \\
		Q_7 & (a,b) =(-2a^2 + 2a)b^9 + (-4a^4 + 2a)b^8 + (2a^6 - 6a^5 - 2a^4 - 6a^3 + 2a^2 + 4a)b^7 \\
			  & \; + (2a^7 - 10a^5 - 14a^3 - 10a^2 - 3a + 2)b^6 + (16a^7 - 8a^5 + 16a^4 + 6a^3 - 2a^2)b^5 \\
			  & \; + (-4a^9 + 40a^7 + 32a^6 + 4a^5 - 12a^4)b^4 + (-24a^9 - 60a^8 - 26a^7 + 4a^6)b^3 \\
			  & \; + (4a^{11} + 28a^{10} + 32a^9 + 18a^8)b^2 + (-4a^{12}-14a^{11}-12a^{10})b + (2a^{13} + 2a^{12}), \\
		Q_8 & (a,b) =2b^6 + (-a^2 - 3a)b^5 + (-10a^4 - 7a^3 + 2a)b^4 + (-13a^6 - 7a^5 + 2a^4 + 2a^3 - 2a^2)b^3 \\
			  & \; + (2a^8 + 4a^7 - 4a^5 + 8a^3 + 4a^2)b^2 + (2a^{10} + 2a^9 - 12a^7 + 12a^5 + 8a^4 + 4a^3 + 4a^2 + 2a)b \\
			  & \; + (-4a^9 + 4a^7 + 4a^6 + 4a^5 + 4a^4 + 2a^3), \\
		Q_9 & (a,b) =-2b^6 + (7a^2 + a)b^5 + (10a^4 + 5a^3 - 4a^2 - 2a)b^4 \\
			  & \; + (-5a^6 - 7a^5 + 2a^4 + 6a^3 - 2a^2 - 2a)b^3 + (-6a^8 - 4a^7 + 12a^5)b^2 \\
			  & \; + (2a^{10} + 2a^9 + 4a^7 + 4a^5 - 4a^3 - 4a^2 - 2a)b + (-4a^9 + 4a^7 - 4a^5 - 4a^4 - 2a^3), \\ 
		Q_{10} & (a,b) =2ab^{12} + (-8a^3 - 2a^2)b^{11} + (-2a^5 - 4a^4 - 2a^3 + 8a^2)b^{10} \\
				 & \; + (2a^7 + 2a^6 - 4a^5 - 14a^4 + 6a^3 + 10a^2 + 2a)b^9 \\
				 & \; + (6a^7 + 8a^6 - 4a^5 - 20a^4 + 10a^3 - 2a^2 + 2a + 2)b^8 \\
				 & \; + (-2a^9 - 2a^8 + 14a^7 + 28a^6 - 6a^5 + 24a^4 + 18a^3 - 9a - 2)b^7 \\
				 & \; + (-10a^9 - 12a^8 + 26a^7 - 4a^6 - 94a^5 - 90a^4 + 40a^3 + 17a^2 + 13a)b^6 \\
				 & \; + (2a^{11} + 2a^{10} - 32a^9 - 58a^8 + 100a^7 + 218a^6 + 36a^5 - 42a^4 - 109a^3 - 12a^2 - 3a + 2)b^5 \\
				 & \; + (14a^{11} + 54a^{10} - 6a^9 - 184a^8 - 200a^7 - 22a^6 + 112a^5 + 64a^4 + 35a^3 - 7a^2 - 3a - 2)b^4 \\
				 & \; + (-2a^{13} - 18a^{12} - 36a^{11} + 44a^{10} + 163a^9 + 108a^8 + 48a^7 - 48a^6 - 101a^5 \\
				 & \quad - 38a^4 + 18a^3 + 12a^2)b^3 \\
				 & \; + (2a^{14} + 16a^{13} + 14a^{12} - 40a^{11} - 83a^{10} - 91a^9 - 14a^8 + 62a^7 + 75a^6 + 27a^5 + 4a^4)b^2 \\
				 & \; + (-2a^{15} - 8a^{14} - a^{13} + 26a^{12} + 31a^{11} + 20a^{10} - a^9 - 42a^8 - 39a^7 - 12a^6)b \\
				 & \; + (a^{16} + a^{15} - 3a^{14} - 3a^{13} - 4a^{12} - 3a^{11} + 7a^{10} + 9a^{9} + 3a^{8}), \\
		Q_{11} & (a,b) =-2b^6 + (5a^2 + 5a)b^5 + (14a^4 - 6a^3 - 2a)b^4 + (-9a^6 - 9a^5 - 12a^4 + 8a^3 + 2a^2 - 2a)b^3 \\
				 & \; + (-2a^8 + a^7 + 26a^5 - 4a^4 - 2a^3 + 4a^2)b^2 + (a^{10} + a^9 + 2a^8 - 2a^7 - 2a^6 + 10a^5 \\
				 & \quad + 2a^4 - 8a^3 - 4a^2)b \\
				 & \; + (-4a^9 + 4a^7 - 4a^5 - 4a^4 - 4a^3 - 4a^2 - 2a). \\
		\end{align*}

\end{document}